\documentclass[12pt]{amsart}

\usepackage{latexsym}
\usepackage{amssymb} 
\usepackage{mathrsfs}
\usepackage{latexsym}
\usepackage{delarray}
\usepackage{amssymb,amsmath,amsfonts,amsthm,mathrsfs}

\renewcommand{\theequation}%
    {\thesection.\arabic{equation}}
\newtheorem{thm}{Theorem}[section]
\newtheorem{lem}[thm]{Lemma}
\newtheorem{prop}[thm]{Proposition}

\newtheorem{exam}[thm]{Example}

\newcommand\R{{\mathbb R}}

\newcommand\Rn{{{\mathbb R}^n}}

\def\p#1{{\left({#1}\right)}}
\def\jp#1{{\left\langle{#1}\right\rangle}}

\newcommand\va{\varphi}

\newcommand\pa{\partial}

\title[Global well-posedness of Kirchhoff systems]
{Global well-posedness of Kirchhoff systems}

\author[Tokio Matsuyama]{Tokio Matsuyama}
\address{
Tokio Matsuyama:
 \endgraf
Department of Mathematics
\endgraf
Faculty of Science and Engineering
\endgraf
Chuo University
\endgraf
1-13-27, Kasuga, Bunkyo-ku
\endgraf
Tokyo 112-8551
\endgraf
Japan
\endgraf
{\it E-mail address} {\rm tokio@math.chuo-u.ac.jp}
}
\author[Michael Ruzhansky]{Michael Ruzhansky}
\address{
  Michael Ruzhansky:
  \endgraf
  Department of Mathematics
  \endgraf
  Imperial College London
  \endgraf
  180 Queen's Gate, London SW7 2AZ
  \endgraf
  United Kingdom
  \endgraf
  {\it E-mail address} {\rm m.ruzhansky@imperial.ac.uk}
  }

\thanks{
The first author was supported by 
Grant-in-Aid for Scientific 
Research (C) (No. 21540198), 
Japan Society for the Promotion of Science. 
The second author was supported in parts by the
EPSRC grant EP/E062873/1 and EPSRC Leadership Fellowship EP/G007233/1.
}
\date{\today}

\subjclass[2010]{Primary 35L40, 35L30; Secondary 35L10, 35L05, 35L75;}
\keywords{Kirchhoff systems; Kirchhoff equation; hyperbolic systems; asymptotic integration}

\begin{document}

\maketitle

\begin{abstract}
The aim of this paper is to establish the $H^1$ global well-posedness for Kirchhoff systems. 
The new approach to the construction of solutions is based on 
the asymptotic integrations for strictly hyperbolic systems 
with time-dependent coefficients. These integrations play 
an important r\^ole to setting the subsequent fixed point argument. 
The existence of solutions for less regular data is discussed, and several examples
and applications are presented.
\end{abstract}

\section{Introduction}
\label{sec:1}
The Kirchhoff equations of the form
\begin{equation}\label{KE}
\partial_t^2 u-a\left(\int_{\mathbb{\mathbb{R}}^n} |\nabla u|^2 dx\right) \Delta u=0
\quad (t\in \R, \, x\in \Rn)
\end{equation}
have been previously considered for various positive functions
$a(s)$. 
Bernstein first studied the global existence for real analytic data (see \cite{Bernstein}), 
and after him, many authors  investigated these equations further
(see \cite{Dancona1,Dancona2,Dancona3,Greenberg,Kajitani,Nishida,
Rzymowski,Yamazaki}). 
Also, the global existence for quasi-analytic data was studied by 
Nishihara (see \cite{Nishihara}), and variants of the class in \cite{Nishihara} were 
discussed in \cite{Ghisi,Hirosawa,Manfrin-JDE}. 

The approach of this paper yields new results already in the scalar
case of the classical Kirchhoff equation \eqref{KE} but, in fact, we are able
to make advances for coupled equations as well or, more generally, 
for Kirchhoff systems.
To this end, Kirchhoff systems are of interest but present several major complications
compared to the scalar case. 
First of all, even for the linearised system, 
it is much more difficult to find a suitable representation of solutions which
would, on one hand, work with the low regularity ($C^1$) of coefficients
while, on the other hand, allow one to obtain sufficiently good estimates for
solutions. Moreover, in the case of systems of higher order, it is impossible to
find its characteristics explicitly, and the geometry of the system or rather of
the level sets of the characteristics enters the picture. 

The main new idea (even for the classical equation \eqref{KE})
behind this paper is to approach the problem by
developing the ``asymptotic integration'' method for the linearised equation
to be able to control its solutions to the extent of being able to prove 
a-priori estimates necessary for the handling of the fully nonlinear problem.
Thus, for the linear strictly
hyperbolic systems we developed
in \cite{MR06} 
the method of asymptotic
integrations allowing us to obtain representations of solutions under
the low regularity of the coefficients. Consequently, we apply it in the present
setting to set
up a suitable fixed point argument assuring the well-posedness of the
Cauchy problem. The results presented in this paper
resolve the well-posedness problem for a 
general class of strictly hyperbolic systems.
Moreover, even for the classical Kirchhoff-type equation \eqref{KE} we obtain
new results. Thus, the regularity of data in Theorem 
\ref{thm:NONLOCAL} is lower than 
that in \cite{Manfrin,Dancona1}. 
Moreover, we prove the well-posedness in low dimensions
$n=1,2$ which remained open since D'Ancona and Spagnolo 
\cite{Dancona1}.

We consider the
Kirchhoff-type systems 
of the form
\begin{equation}\label{EQ:Kirchhoff systems}
\left\{
\begin{aligned}
& D_tU=A(s(t), D_x)U, \quad t\ne 0, \quad x\in \Rn,\\
& U(0,x)=U_0(x)={}^T \left(f_0(x),f_1(x),\ldots,f_{m-1}(x)\right), 
\quad x \in \mathbb{R}^n, 
\end{aligned}\right.
\end{equation}
where $D=-i\pa$ and $A(s,D_x)$ is a first order $m\times m$ 
pseudo-differential 
system with a suitably smooth behaviour in  $s\in \R$ in a neighbourhood of 
$0$; $s(t)$ is a quadratic form defined to be  
\begin{equation}\label{EQ:nonlocal}
s(t)=\langle SU(t,\cdot),U(t,\cdot) \rangle_{\mathbb{L}^2(\mathbb{R}^n)}
=\int_{\mathbb{R}^n} 
{}^T(SU(t,x)) \overline{U(t,x)} \, dx
\end{equation} 
for some $m\times m$ Hermitian matrix $S$, and we put
$\mathbb{L}^2(\mathbb{R}^n)=(L^2(\Rn))^m.$
We allow the operator $A(s,D_x)$ to be pseudo-differential since we 
want our analysis to be applicable to scalar higher order equations and to
coupled equations of higher orders e.g. to coupled Kirchhoff
equations, see Example \ref{exam1} and Example \ref{exam2}.
The precise meaning of ``pseudo-differential'' in this context will be
specified below. 

We assume that the system \eqref{EQ:Kirchhoff systems} is 
strictly hyperbolic. Namely, 
the characteristic polynomial of the differential operator 
$D_t-A(s,D_x)$ has real and distinct roots 
$\varphi_1(s,\xi),\ldots,\varphi_m(s,\xi)$ for any $s$ in the domain of the definition 
of matrices $A(s,\xi)$ and for any $\xi\in \Rn\backslash0$, i.e.
\begin{equation}
\mathrm{det}(\tau I-A(s,\xi))=(\tau-\varphi_1(s,\xi)) \cdots 
(\tau-\varphi_m(s,\xi)).
\label{Kirchhoff strict hyperbolicity1}
\end{equation} 
We assume that $A(s,\xi)$ is positively homogeneous in $\xi$ of
order one, i.e. we have $A(s,\lambda\xi)=\lambda A(s,\xi)$ for all
$s\in [0,\delta]$ for a suitable (usually sufficiently small) $\delta>0$, 
$\xi\not=0$, and $\lambda>0$.
Then by the strict hyperbolicity, 
we may assume that 
\begin{equation}\label{Kirchhoff strict hyperbolicity2}
\inf_{s \in[0,\delta],\, |\xi|=1} 
|\varphi_j(s,\xi)-\varphi_k(s,\xi)|>0 \quad \text{for $j \ne k$.}
\end{equation}
In the last part of the next section some examples of \eqref{EQ:Kirchhoff systems} 
will be presented (see Examples \ref{exam1}, \ref{exam2} and \ref{exam3}). 

One of the main difficulties in the analysis of Kirchhoff equations
is that even if the natural $H^1$ well-posedness
of \eqref{EQ:Kirchhoff systems} holds, this would mean that the function 
$s(t)$ is at most $C^1$. Consequently, even in the case of a linear system
\eqref{EQ:Kirchhoff systems} we would need to analyse systems with low 
regularity $C^1$ of the coefficients, in which case dispersive and Strichartz
estimates are more difficult to obtain due to the lack of a satisfactory
representation for solutions. Such analysis with applications to
nonlinear perturbations and scattering for systems \eqref{EQ:Kirchhoff systems} 
will appear elsewhere.

The global existence of \eqref{EQ:Kirchhoff systems} for differential
systems was analysed by 
Callegari \& Manfrin  (see \cite{Manfrin}, and also \cite{Manfrin1}), 
where the Cauchy data are either smooth compactly supported or
belong to a certain special class 
$\mathscr{M}(\Rn)$, 
which contains a weighted Sobolev space,
and the system is $C^2$ in time.
Precise definition of this class will be introduced as a remark after 
the statement of Theorem \ref{thm:Kirchhoff}. 
The purpose in the present paper is to find global solutions to 
\eqref{EQ:Kirchhoff systems} for a more 
general class of data, removing the smooth compactly supported data
assumption in the general result.
Our approach is to employ asymptotic integrations for {\em linear} hyperbolic 
systems with time-dependent coefficients in order to 
derive a certain integrability of the time-derivative of coefficients 
(see Lemma \ref{lem:Kirchhoff} in \S \ref{sec:3}), which enables us to use 
the fixed point argument. 
It should be noted that the present argument will also resolve an open problem 
of the well-posedness in low dimensions
in D'Ancona \& Spagnolo \cite{Dancona1} (see Theorem \ref{thm:NONLOCAL}).

This paper is organised as follows. 
In \S \ref{sec:results} we formulate the main result and give examples and
several corollaries for equations of different types.
In \S\ref{sec:2} 
we will introduce asymptotic 
integrations for linear hyperbolic systems with time-dependent coefficients, 
which enable us to prove Theorem \ref{thm:Kirchhoff}. The proof of 
Theorem \ref{thm:Kirchhoff} will be given in \S \ref{sec:3}.


\section{Global well-posedness for Kirchhoff equations and systems}
\label{sec:results}

To state the main result, let us introduce a class of data which ensures the 
global well-posedness for \eqref{EQ:Kirchhoff systems}. 
A class $\mathscr{Y}(\mathbb{\mathbb{R}}^n)$ 
consists of all $U_0={}^T(f_0,f_1,\ldots,f_{m-1})\in (\mathscr{S}^\prime(\Rn))^m$ such that 
\[
\pmb{|}U_0\pmb{|}_{\mathscr{Y}(\mathbb{R}^n)}:= \sum_{j,k=0}^{m-1}
\int^\infty_{-\infty}\left(
\int_{\mathbb{S}^{n-1}}\left|\int_0^\infty 
e^{i\tau \rho}
\widehat{f}_j(\rho\omega)\overline{\widehat{f}_k(\rho\omega)} 
\rho^n \, d\rho \right|\, d\sigma(\omega)\right)
\, d\tau<\infty,
\]
where $\mathscr{S}^\prime(\Rn)$ is the space of tempered distributions on $\Rn$, and 
$\mathbb{S}^{n-1}$ is $(n-1)$-dimensional sphere and 
$d\sigma(\omega)$ is the $(n-1)$-dimensional Hausdorff measure.

We denote $\mathbb{H}^\sigma(\mathbb{\mathbb{R}}^n)
=(H^\sigma(\mathbb{\mathbb{R}}^n))^m$
for $\sigma\in \mathbb{\mathbb{R}}$, 
where $H^\sigma(\mathbb{\mathbb{R}}^n)
=\langle D \rangle^{-\sigma} L^2(\mathbb{\mathbb{R}}^n)$ 
are the standard Sobolev spaces, and $\jp{D}=(1-\Delta)^{1/2}$. 
The space $\mathbb{H}^\sigma_\varkappa(\mathbb{R}^n)$ denotes the
$m$ direct product of weighted Sobolev spaces $H^\sigma_\varkappa(\mathbb{R}^n)$, 
which consist of all $f\in \mathscr{S}^\prime(\mathbb{R}^n)$
such that $\langle x \rangle^\varkappa f 
\in H^\sigma(\mathbb{R}^n)$, and $\jp{x}=(1+|x|^{2})^{1/2}$. 
Then by using Lemma A.1 in \cite{Dancona2}, 
we conclude that 
\begin{equation}\label{EQ:weight}
\mathbb{H}^1_\varkappa(\mathbb{R}^n)
\subset \mathscr{Y}(\mathbb{R}^n),
\quad \forall \varkappa>1.
\end{equation}
For a function, say $\varphi(\xi)$, positively homogeneous of order one in $\xi$,
we can factor out $\xi$ and restrict the function $\varphi$ to the unit sphere;
in this case we will be using the notations like
$\varphi(\xi/|\xi|)\in L^{\infty}(\Rn\backslash 0)$ instead of $L^{\infty}(\Rn)$, since
an extension of $\varphi(\xi/|\xi|)$ to $\xi=0$ is irrelevant for our analysis.

\medskip
We shall prove here the following: 
\begin{thm}\label{thm:Kirchhoff}
Assume that system \eqref{EQ:Kirchhoff systems} is strictly hyperbolic, and that 
$A(s,\xi)=(a_{jk}(s,\xi))_{j,k=1}^m$ is an $m\times m$ matrix, positively
homogeneous of order one in $\xi$, whose entries $a_{jk}(s,\xi/|\xi|)$ are in 
$\mathrm{Lip}([0,\delta];L^\infty(\mathbb{R}^n\backslash0))$ 
for some $\delta>0$. 
If $U_0\in \mathbb{L}^2(\mathbb{R}^n)\cap \mathscr{Y}(\mathbb{R}^n)$ satisfies
\begin{equation}\label{EQ:small}
\|U_0\|^2_{\mathbb{L}^2(\mathbb{R}^n)}+\pmb{|}U_0\pmb{|}_{\mathscr{Y}(\mathbb{R}^n)}
\ll 1,
\end{equation}
then system \eqref{EQ:Kirchhoff systems}-\eqref{EQ:nonlocal} has a solution 
$U(t,x)\in C(\mathbb{R};\mathbb{L}^2(\mathbb{R}^n))$.
In addition to \eqref{EQ:small}, if $U_0\in \mathbb{H}^1(\mathbb{R}^n)$,
then the solution $U(t,x)$ exists uniquely in the class 
$C(\mathbb{R};\mathbb{H}^1(\mathbb{R}^n))\cap C^1(\mathbb{R};\mathbb{L}^2(\mathbb{R}^n))$.
\end{thm}

As a related result to Theorem \ref{thm:Kirchhoff}, 
Callegari \& Manfrin introduced the following class (see \cite{Manfrin}): 
\[
\mathscr{M}(\mathbb{R}^n)=\left\{U_0(x)={}^T (f_0(x),f_1(x),\ldots,f_{m-1}(x))\in 
(\mathscr{S}^\prime(\mathbb{R}^n))^m:
\pmb{|}U_0\pmb{|}_{\mathscr{M}(\mathbb{R}^n)}<\infty \right\},
\]
where 
\[
\pmb{|}U_0\pmb{|}_{\mathscr{M}(\mathbb{R}^n)}=\sum_{k=0}^2\sum_{j=0}^{m-1}
\sup_{\omega\in \mathbb{S}^{n-1}} \int^\infty_0
\left|\partial^k_\rho \widehat{f}_j(\rho\omega)\right|^2 \left(1+\rho^{\max\{n,2\}} \right)
\, d\rho.
\] 
Hence, in particular, Theorem \ref{thm:Kirchhoff} generalises \cite{Manfrin} since
the inclusion among this class and ours is:
\[
(C_0^\infty(\mathbb{R}^n))^m\subset 
\mathbb{L}^1_2(\mathbb{R}^n)\cap \mathbb{H}^1_2(\mathbb{R}^n)\subset 
\mathscr{M}(\mathbb{R}^n) \subset 
\mathscr{Y}(\mathbb{R}^n),
\]
where $\mathbb{L}^1_2(\mathbb{R}^n)$ is the $m$ direct product of $L^1_2(\Rn)=\{f\in 
\mathscr{S}^\prime(\Rn):\langle x \rangle^2f\in L^1(\Rn)\}$. \\

Needless to say, Theorem \ref{thm:Kirchhoff} covers the second order case, 
i.e., the Kirchhoff equation
\begin{equation*}
\partial_t^2 u-\left(1+\int_{\mathbb{R}^n} |\nabla u|^2 \, dx\right)\Delta u=0.
\end{equation*}
In this case,
Yamazaki found a general class that ensures global 
well-posedness (see \cite{Yamazaki}). 
In fact, she proved that the space
$H^2_\varkappa(\mathbb{R}^n) \times H^1_\varkappa(\mathbb{R}^n)$ 
introduced by D'Ancona \& Spagnolo (see \cite{Dancona2}) is 
contained in $\mathscr{Y}_\varkappa(\mathbb{R}^n)$ for any $\varkappa\in (1,n+1]$. 
The class $\mathscr{Y}_\varkappa(\mathbb{R}^n)$ consists of 
the pairs of data $(f_0,f_1)\in H^{3/2}(\mathbb{R}^n)\times H^{1/2}(\mathbb{R}^n)$ such that
\[
\sum_{j,k=0}^1 \sup_{\tau\in \mathbb{R}}\langle \tau\rangle^\varkappa
\left|\int_{\mathbb{R}^n} e^{i\tau|\xi|} \widehat{f}_j(\xi)\overline{\widehat{f}_k(\xi)}
|\xi|^{3-j-k}\, d\xi \right|<\infty.
\]
After her, Kajitani found the most general class $\mathscr{K}(\mathbb{R}^n)$:
\[
\mathscr{K}(\mathbb{R}^n)=\left\{
(f_0,f_1)\in H^{3/2}(\mathbb{R}^n)\times H^{1/2}(\mathbb{R}^n):
\pmb{|}(f_0,f_1)\pmb{|}_{\mathscr{K}(\mathbb{R}^n)}<\infty
\right\},
\]
where 
\[
\pmb{|}(f_0,f_1)\pmb{|}_{\mathscr{K}(\mathbb{R}^n)}=\sum_{j,k=0}^1 \int^\infty_{-\infty}
\left|\int_{\mathbb{R}^n} e^{i\tau|\xi|} \widehat{f}_j(\xi)\overline{\widehat{f}_k(\xi)}
|\xi|^{3-j-k}\, d\xi \right| \, d\tau
\]
(see \cite{Kajitani}, and 
also Rzymowski \cite{Rzymowski} who considered the one-dimensional case). 
As to the exterior version of the class
$\mathscr{K}(\mathbb{R}^n)$, 
we can refer to the recent results \cite{Ext-Matsuyama,Rend-Matsuyama} 
(see also \cite{Yamazaki1,Yamazaki2}). 
The inclusions among these classes and ours 
are: 
\[
H^2_\varkappa(\mathbb{R}^n) \times H^1_\varkappa(\mathbb{R}^n)
\subset 
\left\{
\begin{aligned}
& \mathscr{Y}_\varkappa(\mathbb{R}^n)\\
& \mathscr{Y}(\mathbb{R}^n)\\
\end{aligned}
\right.
\subset \mathscr{K}(\mathbb{R}^n).
\]
Here the first inclusion holds for $\varkappa\in (1,n+1]$
and the second one 
is valid for any $\varkappa>1$.\\

In the rest of this section, let us give some examples of applications
of our result. 
First of all, we note that Theorem \ref{thm:Kirchhoff} covers all the examples of 
Callegari and Manfrin \cite{Manfrin}, 
in particular, the Kirchhoff equations of higher order etc. There, 
it is assumed that the Cauchy 
data $f_k(x)$, $k=0,1,\ldots,m-1$, 
belong to $\mathscr{M}(\mathbb{R}^n)$, or even $C_0^\infty(\mathbb{R}^n)$. 
More precisely, we have: 
\begin{exam} \label{exam1}
Let us consider the Cauchy problem 
\begin{equation}\label{Kirchhoff-Equation}
\left\{ 
\begin{aligned} 
& L\left(D_t,D_x,s(t)\right)u
\equiv D^{m}_t u+\sum_{\underset{j \le m-1}
{\vert \nu \vert+j=m}} 
b_{\nu,j}\left(s(t)\right) 
D^{\nu}_x D^{j}_t u=0, \\
& D^{k}_t u(0,x)=f_k(x), \quad 
k=0,1,\cdots,m-1. 
\end{aligned}\right.
\end{equation}
Here the quadratic form $s(t)$ is given by 
\[
s(t)=\int_{\mathbb{R}^n}\sum_{|\beta|=|\gamma|=m-1}s_{\beta \gamma}D^\beta u(t,x)
\overline{D^\gamma u(t,x)}\, dx,
\]
where $\beta=(\beta_t,\beta_x)$, $\gamma=(\gamma_t,\gamma_x)$, 
$D^\beta=D_t^{\beta_t}D_x^{\beta_x}$ and 
$s_{\beta \gamma}=\overline{s_{\gamma\beta}}$. 
We assume that the symbol $L(\tau,\xi,s)$ of the differential operator 
$L(D_t,D_x,s)$ has real 
and distinct roots 
$\varphi_1(s,\xi),\ldots,\varphi_m(s,\xi)$ 
for $\xi \ne0$ and $0\le s \le \delta$ with $\delta>0$, i.e., 
\begin{equation} \label{EQ:Strict-1}
L(\tau,\xi,s)=(\tau-\varphi_1(s,\xi)) \cdots 
(\tau-\varphi_m(s,\xi)), 
\end{equation}
\begin{equation} \label{EQ:Strict-2}
\inf_{s \in[0,\delta],\, |\xi|=1} 
|\varphi_j(s,\xi)-\varphi_k(s,\xi)|>0
\quad \text{for $j \ne k$.}
\end{equation}
By taking the Fourier transform in the space variables and introducing the vector 
\[
V(t,\xi)={}^T(\vert \xi \vert^{m-1} \widehat{u}(t,\xi),\vert \xi \vert^{m-2} D_t \widehat{u}(t,\xi),
\cdots,D^{m-1}_t \widehat{u}(t,\xi)),
\]
we reduce the problem to the  system 
\begin{align*}
D_t V
=& \left(
\begin{array}{cccc}
0 & 1 & \ldots & 0 \\ 
0 & 0 & \ddots & 0 \\
\vdots & \ddots & \ddots & 1 \\
-H_m(s(t),\xi) 
& -H_{m-1}(s(t),\xi) 
& \ldots 
& -H_1(s(t),\xi) 
\end{array}\right)|\xi|V\\
=&A(s(t),\xi)V,
\end{align*}
where we put
\[
H_j(s(t),\xi)=\sum_{\vert \nu \vert=j} b_{\nu,m-j}(s(t)) (\xi/|\xi|)^\nu,
\quad (j=1,\ldots,m).
\]
Then we have: 
\begin{thm} \label{thm:theorem 1.3}
Assume \eqref{EQ:Strict-1}--\eqref{EQ:Strict-2}. 
If $f_k\in H^{m-k}(\mathbb{R}^n)$ $(k=0,1,\ldots,m-1)$, then 
\eqref{Kirchhoff-Equation} has a unique solution 
$u(t,x)\in \bigcap_{k=0}^{m-1} C^k(\mathbb{R};H^{m-k}(\mathbb{R}^n))$, provided that 
the quantity 
$$\|(|D_x|^{m-1}f_0,|D_x|^{m-2}f_1,\ldots,f_{m-1})\|^2_{\mathbb{L}^2(\mathbb{R}^n)}+
\pmb{|} (|D_x|^{m-1}f_0,|D_x|^{m-2}f_1,\ldots,f_{m-1})\pmb{|}_{\mathscr{Y}(\mathbb{R}^n)}
$$ 
is sufficiently small. 
\end{thm}
\end{exam}

As a new example of \eqref{EQ:Kirchhoff systems}, 
we can treat the completely coupled Kirchhoff equations of
the following type. 
\begin{exam}\label{exam2}
Let us consider the Cauchy problem 
\begin{equation}\label{EQ:Example}
\left\{ 
\begin{aligned}
& \partial_t^2u
-a_1\left(1+\|\nabla u(t)\|_{L^2}^2+\|\nabla v(t)\|_{L^2}^2\right)
\Delta u+P_1(t,D_x)v=0, \\ 
& \partial_t^2v
-a_2\left(1+\|\nabla u(t)\|_{L^2}^2+\|\nabla v(t)\|_{L^2}^2\right)
\Delta v+P_2(t,D_x)u=0,\\
& \partial_t^j u(0,x)=u_j(x), 
\quad \partial_t^j v(0,x)=v_j(x), \quad 
j=0,1,
\end{aligned}
\right.
\end{equation}
for some second order homogeneous polynomials $P_1(t,D_x),P_2(t,D_x)$, and 
for some constants $a_1,a_2>0$ with $a_1\ne a_2$.
The quadratic form is given here by
\[
s(t)=\|\nabla u(t)\|_{L^2}^2+\|\nabla v(t)\|_{L^2}^2.
\] 
We assume that
\begin{equation}\label{EQ:Ass-AP1}
|\xi|^{-2}P_k(t,\xi)\in \mathrm{Lip}_{\mathrm{loc}}(\mathbb{R};L^\infty(\mathbb{R}^n\backslash0)), \quad 
|\xi|^{-2}\partial_t P_k(t,\xi)\in L^1(\mathbb{R};L^\infty(\mathbb{R}^n\backslash0)) 
\end{equation}
for $k=1,2$, and that
\begin{equation}
\inf_{t\in \R,\, |\xi|=1} 
\left\{(a_1-a_2)^2 +4P_1(t,\xi)P_2(t,\xi)\right\}>0, 
\end{equation}
\begin{equation}\label{EQ:Ass-AP2}
\inf_{t\in \R,\, |\xi|=1} 
\left\{a^2_1 a^2_2-P_1(t,\xi)P_2(t,\xi)\right\}>0.
\end{equation}
By taking the Fourier transform in the space variables and introducing the vector 
\[
V(t,\xi)
={}^T(|\xi|\widehat{u}(t,\xi),\widehat{u}^\prime(t,\xi),
|\xi|\widehat{v}(t,\xi),\widehat{v}^\prime(t,\xi)),
\]
we rewrite \eqref{EQ:Example} as a system 
\[
D_t V= 
\begin{pmatrix}
0&-i|\xi|&0&0\\
ic_1(t)^2|\xi|&0& iP_1(t,\xi)|\xi|^{-1}&0\\
0&0&0& -i|\xi|\\
iP_2(t,\xi)|\xi|^{-1} &0& ic_2(t)^2|\xi|&0
\end{pmatrix}V
=: A(s(t),\xi)V, 
\]
where 
\[
c_k(t)=\sqrt{a_k(1+s(t))}, \quad k=1,2.
\]
The four characteristic roots of the equation
$$\mathrm{det}(\tau I-A(s(t),\xi))=0$$ 
in $\tau$
are given by
\begin{multline*}
\varphi_{1,2,3,4}(s(t),\xi)=\\
\pm \frac{|\xi|}{\sqrt{2}}
\sqrt{c_1(t)^2+c_2(t)^2\pm
\sqrt{\{c_1(t)^2-c_2(t)^2\}^2+4P_1(t,\xi) P_2(t,\xi)|\xi|^{-4}}}.
\end{multline*}
Then it follows from \eqref{EQ:Ass-AP1}--\eqref{EQ:Ass-AP2} that 
\[
\inf_{s \in[0,\delta],\, |\xi|=1} 
|\varphi_j(s,\xi)-\varphi_k(s,\xi)|>0
\quad \text{for $j \ne k$.}
\]
\end{exam}

Thus we have the following: 
\begin{thm} Assume \eqref{EQ:Ass-AP1}--\eqref{EQ:Ass-AP2}. If 
$(u_j,v_j)\in \mathbb{H}^{2-j}(\mathbb{R}^n)$ for $j=0,1$, 
then 
\eqref{EQ:Example} has a pair of unique solutions 
$(u,v)\in \bigcap_{k=0,1} 
C^k(\mathbb{R};\mathbb{H}^{2-k}(\mathbb{R}^n))$ 
provided that the quantity 
$$\|(|D_x|u_0,u_1,|D_x|v_0,v_1)\|^2_{\mathbb{L}^2(\mathbb{R}^n)}+
\pmb{|} (|D_x|u_0,u_1,|D_x|v_0,v_1)\pmb{|}_{\mathscr{Y}(\mathbb{R}^n)}$$
is sufficiently small. 
\end{thm}

Theorem \ref{thm:Kirchhoff} can be also generalised in another direction. 
In fact, as it is pointed out in \cite{Manfrin}, the nonlocal term \eqref{EQ:nonlocal} 
may be replaced by 
\[
s(t)=\left\langle |\xi|^{-k}S\widehat{U}(t,\xi),\widehat{U}(t,\xi) 
\right\rangle_{\mathbb{L}^2(\mathbb{R}^n)}
\]  
for $0\le k\le n-1$. By this little change we can generalize Theorem \ref{thm:theorem 1.3} 
without any change in the proof. More precisely, we have the following example, 
which resolves an open problem in D'Ancona \& Spagnolo \cite{Dancona1}. 

\begin{exam} \label{exam3}
Let us consider the Cauchy problem for the second order equation of the form 
\begin{equation}\label{EQ:Spagnolo}
\partial^2_t u-\left(1+\int_{\mathbb{R}^n} |u(t,x)|^2 \, dx \right)\Delta u=0, 
\quad t\ne0, \quad x \in \mathbb{R}^n,
\end{equation}
with data 
\begin{equation} \label{EQ:Spagnolo-data}
u(0,x)=f_0(x), \quad \partial_t u(0,x)=f_1(x).
\end{equation}
In this particular case, the nonlocal term $s(t)$ is defined by 
\[
s(t)=\|u(t)\|^2_{L^2(\mathbb{R}^n)}.
\]
Introducing another class of data 
\[
\widetilde{\mathscr{Y}}(\mathbb{R}^n)
=\left\{(f_0,f_1)\in \mathscr{S}^\prime(\mathbb{R}^n)\times \mathscr{S}^\prime(\mathbb{R}^n):
\pmb{|}(f_0,f_1)\pmb{|}_{\widetilde{\mathscr{Y}}(\mathbb{R}^n)}<\infty\right\},
\]
where we put 
\[
\pmb{|}(f_0,f_1)\pmb{|}_{\widetilde{\mathscr{Y}}(\mathbb{R}^n)}= \sum_{j,k=0}^1
\int^\infty_{-\infty}\left(
\int_{\mathbb{S}^{n-1}}\left|\int_0^\infty
e^{i\tau \rho}
\widehat{f}_j(\rho\omega)\overline{\widehat{f}_k(\rho\omega)} 
\rho^{n-j-k} \, d\rho \right|\, d\sigma(\omega)\right)
\, d\tau,
\]
we have{\rm :}
\begin{thm} \label{thm:NONLOCAL}
Let $n\ge1$. Then, for any 
$(f_0,f_1)\in (H^1(\mathbb{R}^n)\times L^2(\mathbb{R}^n))\cap \widetilde{\mathscr{Y}}(\mathbb{R}^n)$, 
\eqref{EQ:Spagnolo}--\eqref{EQ:Spagnolo-data} has a unique solution 
$u\in \bigcap_{k=0,1}C^k(\mathbb{R};H^{1-k}(\mathbb{R}^n))$, provided that the quantity 
\[
\|f_0\|^2_{L^2(\mathbb{\mathbb{R}}^n)}+\|f_1\|^2_{\dot{H}^{-1}(\mathbb{\mathbb{R}}^n)}+
\pmb{|}(f_0,f_1)\pmb{|}_{\widetilde{\mathscr{Y}}(\mathbb{\mathbb{R}}^n)}
\]
is sufficiently small. Here $\dot{H}^{-1}(\mathbb{R}^n)=|D_x|L^2(\mathbb{R}^n)$ is 
the homogeneous Sobolev space of order $-1$. 
\end{thm}
\end{exam}

Let us give a few remarks on Theorem \ref{thm:NONLOCAL}. This theorem generalises the
results of 
\cite{Manfrin} and \cite{Dancona1}. Indeed, when the space dimension $n$ is greater than 2, 
$n\geq 3$,
a similar result was obtained in \cite{Manfrin} and \cite{Dancona1}. However, 
the regularity of data in Theorem \ref{thm:NONLOCAL} is lower than 
that in \cite{Manfrin,Dancona1}. 
It should be noted that Theorem \ref{thm:NONLOCAL} 
also covers low dimensions $n=1,2$, the case that remained open since 
\cite{Manfrin,Dancona1}.


\section{Asymptotic integrations for linear hyperbolic system}
\label{sec:2}
In this section we shall derive asymptotic integrations for 
linear hyperbolic systems with time-dependent coefficients, a kind of  
representation formula for their solutions. In fact, we have discussed such arguments in 
our recent paper \cite{MR-MN} in the context of the scattering
problems. To make the argument self-contained, 
we must give the proof completely, because the Fourier integral form of solutions $U$ 
to Kirchhoff system \eqref{EQ:Kirchhoff systems} will be obtained by 
a careful observation of  
the construction of asymptotic integrations for linear systems. 
We note that the major advantage of the asymptotic integration method
developed in \cite{MR06} in comparison to other approaches, e.g. the
diagonalisation schemes for systems as in \cite{RW}, is the low $C^1$ regularity of
coefficients in $t$ sufficient for the construction compared to higher
regularity required for other methods. \\

Let us consider the linear Cauchy problem
\begin{equation} \label{Eq}
\left\{
\begin{aligned}
& D_t U=A(t,D_x)U, \quad (t,x) \in \mathbb{R} \times \mathbb{R}^n,\\ 
&U(0,x)={}^T \left(f_0(x),f_1(x),\ldots,f_{m-1}(x)\right)
\in (C^\infty_0(\Rn))^m,
\end{aligned}\right.
\end{equation} 
where $A(t,D_x)$ is a first order $m\times m$ pseudo-differential system,
with symbol $A(t,\xi)$ of the form
$A(t,\xi)=(a_{ij}(t,\xi))_{i,j=1}^m$, positively homogeneous of order
one in $\xi$. We assume that 
\begin{equation}\label{hyp1}
\text{$a_{ij}(t,\xi/|\xi|) \in \mathrm{Lip}_{\mathrm{loc}}(\R;L^\infty(\Rn\backslash0))$
and $\partial_t a_{ij}(t,\xi/|\xi|) \in L^1(\R;L^\infty(\Rn\backslash0))$},
\end{equation}
and that system \eqref{Eq} is strictly hyperbolic: 
\begin{equation}\label{hyp2}
\text{$\mathrm{det} (\tau I-A(t,\xi))=0$ has real and distinct roots 
$\varphi_1(t,\xi),\ldots,\varphi_m(t,\xi)$,}
\end{equation}
i.e., 
\begin{equation} \label{hyp3}
\underset{t\in \mathbb{R},|\xi|=1}{\inf}
|\varphi_j(t,\xi)-\varphi_k(t,\xi)|
>0 \quad \text{for $j\ne k$.} 
\end{equation}
Notice that each characteristic root $\varphi_j(t,\xi)$ is positively
homogeneous of order one in $\xi$. 

Let us first analyse certain basic properties of characteristic roots 
$\varphi_k(t,\xi)$. The next proposition is 
established in \cite{MR-MN}. 

\begin{prop} [\cite{MR-MN} (Proposition 2.1)] \label{prop:root}
Let $D_t-A(t,D_x)$ be a strictly hyperbolic operator as above. 
If $a_{ij}(t,\xi/|\xi|)$ belong to $\mathrm{Lip}_{\mathrm{loc}}(\R;L^\infty(\Rn\backslash0))$
for $i,j=1,\ldots,m$, then $|\varphi_{k}(t,\xi)|\leq C|\xi|$ for some $C>0$, and functions
$\partial_t \varphi_k(t,\xi)$, $k=1,\ldots,m$, are positively
homogeneous of order one in $\xi$. 
In addition, if $\partial_t a_{ij}(t,\xi/|\xi|)$ belong to 
$L^1(\R;L^\infty(\Rn\backslash0))$ for $i,j=1,\ldots,m$, 
then we have also 
$$\partial_t\varphi_k(t,\xi/|\xi|)\in L^1(\R;L^\infty(\Rn\backslash0)).$$ 
\end{prop}
\begin{proof}
Let us show first that $\varphi_k(t,\xi)$ are bounded with respect to 
$t\in \mathbb{R}$, i.e., 
\begin{equation}
\vert \varphi_k(t,\xi) \vert \le C \vert \xi \vert, \quad 
\text{for all $\xi \in \mathbb{R}^n$, $t \in \mathbb{R}$, $k=1,\ldots,m$.}
\label{phibdd}
\end{equation}
We will use the fact that $\varphi_k(t,\xi)$ are roots of the polynomial 
$$L(t,\tau,\xi)=\mathrm{det}(\tau I-A(t,\xi))$$ of the form 
$$L(t,\tau,\xi)=\tau^m+\alpha_1(t,\xi)\tau^{m-1}
+\cdots+\alpha_m(t,\xi)$$ with $|\alpha_j(t,\xi)|\leq M|\xi|^j$, 
for some $M\geq 1$, where 
\begin{eqnarray*}
\alpha_j(t,\xi)=(-1)^j \sum_{i_1<i_2<\cdots<i_j}
\mathrm{det}\left(
\begin{array}{ccc}
a_{i_1 i_1}(t,\xi) 
& \cdots & a_{i_1 i_j}(t,\xi) \\
\vdots & \ddots & \vdots \\
a_{i_j i_1}(t,\xi)
& \cdots & a_{i_j i_j}(t,\xi) \\
\end{array}
\right).
\end{eqnarray*}
Suppose that one of its roots $\tau$ satisfies $|\tau(t,\xi)|>2M|\xi|$. 
Then
\begin{align*}|L(t,\tau,\xi)| & 
 \geq |\tau|^{m} \p{1-\frac{|\alpha_1(t,\xi)|}{|\tau|}-
\cdots-\frac{|\alpha_m(t,\xi)|}{|\tau|^m} }      \\
& > 
2M|\xi|^m \p{1-\frac{1}{2}-\frac{1}{4M}-\cdots-\frac{1}{2^m M^{m-1}}}>0,
\end{align*}
hence $|\tau(t,\xi)|\leq 2M|\xi|$ for all $\xi\in\Rn$. 
Thus we establish \eqref{phibdd}. 

Differentiating \eqref{hyp2} with respect to $t$, we get
$$ 
\frac{\partial L(t,\tau,\xi)}{\partial t}=
\sum_{j=0}^m \partial_t \alpha_{m-j}(t,\xi) \tau^j=
-\sum_{k=1}^m \partial_t\varphi_k(t,\xi)\prod_{r\not=k} 
\p{\tau-\varphi_r(t,\xi)}.
$$
Setting $\tau=\varphi_k(t,\xi)$, we obtain
\begin{equation}\label{2-product}
\partial_t\varphi_k(t,\xi)\prod_{r\not=k} 
\p{\varphi_k(t,\xi)-\varphi_r(t,\xi)}= -
\sum_{j=0}^m \partial_t \alpha_{m-j}(t,\xi) \varphi_k(t,\xi)^j.
\end{equation}
The positive homogeneity of order one of $\partial_t \varphi_k(t,\xi)$ 
is an immediate consequence of \eqref{2-product}. 
Now, by using \eqref{hyp3}, \eqref{phibdd}, 
and the assumption 
that $|\xi|^{-j}\partial_t \alpha_j(\cdot,\xi)\in L^1(\R;L^\infty(\Rn\backslash0))$ for all $j$, 
we conclude that  
$\partial_t\varphi_k(\cdot,\xi/|\xi|)\in L^1(\R;L^\infty(\Rn\backslash0))$ for
$k=1,\ldots,m$.  
The proof is complete. 
\end{proof}

To state the main auxiliary result on the representation of solutions,
we prepare the next lemma. 
\begin{lem}[Mizohata~\cite{Mizohata} (Proposition~6.4)] 
\label{Diagonalisation}
Assume \eqref{hyp1}--\eqref{hyp3}. 
Then there exists a matrix $\mathscr{N}=\mathscr{N}(t,\xi)$ of homogeneous 
order of $0$ in $\xi$ satisfying the following properties{\rm :} 

\noindent 
{\rm (i)} $\mathscr{N}(t,\xi) A(t,\xi/|\xi|)
=\mathscr{D}(t,\xi)\mathscr{N}(t,\xi)$, 
where 
\[
\mathscr{D}(t,\xi)
=\mathrm{diag} \left( \varphi_1(t,\xi/|\xi|),\ldots,
\varphi_m(t,\xi/|\xi|)\right);
\]

\noindent 
{\rm (ii)} $\displaystyle
{\inf_{\xi \in \mathbb{R}^n\backslash 0, t \in \mathbb{R}}}
\vert {\rm det} \, \mathscr{N}(t,\xi)) \vert>0;$

\noindent 
{\rm (iii)} $\mathscr{N}(t,\xi)\in \mathrm{Lip}_{\mathrm{loc}}
(\R;(L^\infty(\Rn \backslash0))^{m^2})$ and 
$\partial_t \mathscr{N}(t,\xi)\in L^1(\mathbb{R};(L^\infty(\Rn\backslash0))^{m^2})$.
\end{lem}

The next proposition is known as Levinson's lemma in the 
theory of ordinary differential equations 
(see Coddington and Levinson \cite{Coddington}, and also 
Hartman \cite{Hartman}); the new feature
here is the additional dependence on $\xi$, which is crucial
for our analysis (see also Proposition 2.3 from \cite{MR-MN} and
Theorem 3.1 from \cite{MR06}). 

\begin{prop} \label{prop:Rep}
Assume \eqref{hyp1}--\eqref{hyp3}. 
Let 
$\mathscr{N}(t,\xi)$ be the diagonaliser of 
$A(t,\xi/|\xi|)$ constructed in Lemma \ref{Diagonalisation}. 
Then there exist vector-valued functions 
$\pmb{a}^j(t,\xi)$, $j=0,1,\ldots,m-1$, 
determined by the initial value problem 
\[
D_t\pmb{a}^j(t,\xi)=C(t,\xi)\pmb{a}^j(t,\xi), \qquad 
\left(\pmb{a}^0(0,\xi),\cdots,\pmb{a}^{m-1}(0,\xi)\right)=\mathscr{N}(0,\xi),
\]
\[
\text{with} \quad C(t,\xi)=\Phi(t,\xi)^{-1} (D_t \mathscr{N}(t,\xi))
\mathscr{N}(t,\xi)^{-1}\Phi(t,\xi) \in L^1(\mathbb{R};(L^\infty(\Rn\backslash0))^{m^2}),
\]
such that 
the solution $U(t,x)$ of \eqref{Eq} is represented by 
\begin{equation}\label{EQ:Representation of U}
U(t,x)=\sum_{j=0}^{m-1} 
\mathscr{F}^{-1} \left[\mathscr{N}(t,\xi)^{-1}\Phi(t,\xi)
\pmb{a}^j(t,\xi)\widehat{f}_j(\xi) \right](x), 
\end{equation}
where $\mathscr{F}^{-1}$ stands for the inverse Fourier transform and we put 
\[
\Phi(t,\xi)
=\mathrm{diag}\left(e^{i\int_0^t\varphi_1(\tau,\xi)\,d\tau},\cdots,
e^{i\int_0^t\varphi_m(\tau,\xi)\,d\tau}\right).
\]
Moreover, the following estimates hold{\rm :}
\begin{equation}\label{EQ:uniform estimate}
\sup_{t\in \R}\left\|\pmb{a}^j(t,\xi)\right\|_{(L^\infty(\Rn\backslash0))^m} 
\le C
\end{equation}
for $j=0,1,\ldots,m-1$. 
\end{prop}
\begin{proof}
Applying the Fourier transform on $\mathbb{R}_x^n$, we get the 
following ordinary differential system from \eqref{Eq}: 
\begin{equation} \label{EQ}
D_t \pmb{v}=A(t,\xi/|\xi|)|\xi|\pmb{v}, \quad (\pmb{v}=\widehat{U}).
\end{equation}
Multiplying \eqref{EQ} by $\mathscr{N}=\mathscr{N}(t,\xi)$ 
from Lemma \ref{Diagonalisation} and putting $\mathscr{N}\pmb{v}=\pmb{w}$, we get 
\begin{equation}
D_t \pmb{w}=\mathscr{D}|\xi|\pmb{w}+(D_t \mathscr{N})\pmb{v}
=\left(\mathscr{D}|\xi|+(D_t \mathscr{N})\mathscr{N}^{-1}\right)\pmb{w},
\label{Ref EQ}
\end{equation}
since $\mathscr{N}A(t,\xi/|\xi|)=\mathscr{D}\mathscr{N}$ by 
Lemma \ref{Diagonalisation}. We can expect that the solutions of 
\eqref{Ref EQ} are asymptotic to some solution of 
\begin{equation}
D_t \pmb{y}=\mathscr{D}|\xi|\pmb{y}.
\label{Ref EQ1}
\end{equation}
Let $\Phi(t,\xi)$ be the fundamental matrix of \eqref{Ref EQ1}, i.e., 
\[
\Phi(t,\xi)=\mathrm{diag}\left( 
e^{i\int_0^t \varphi_1(\tau,\xi) \, d\tau}, 
\cdots, e^{i\int_0^t \varphi_m(\tau,\xi) \, d\tau}\right). 
\]
If we perform the Wronskian transform 
$\pmb{a}(t,\xi)=\Phi(t,\xi)^{-1}\pmb{w}(t,\xi)$, then system 
\eqref{Ref EQ} reduces to the system 
$D_t\pmb{a}=C(t,\xi)\pmb{a}$, where $C(t,\xi)$ is given by 
\[
C(t,\xi)=\Phi(t,\xi)^{-1} (D_t \mathscr{N}(t,\xi))
\mathscr{N}(t,\xi)^{-1}\Phi(t,\xi). 
\]
We note that $C(t,\xi)\in L^1(\R;(L^\infty(\Rn\backslash0))^{m^2})$, since 
$D_t \mathscr{N}(t,\xi)\in L^1(\R;(L^\infty(\Rn\backslash0))^{m^2})$ by 
Lemma \ref{Diagonalisation}.
Since $\pmb{w}(t,\xi)=\Phi(t,\xi)\pmb{a}(t,\xi)$ and 
$\mathscr{N}(t,\xi)\pmb{v}(t,\xi)=\pmb{w}(t,\xi)$, we get 
\[
\pmb{v}(t,\xi)=\mathscr{N}(t,\xi)^{-1}\Phi(t,\xi)\pmb{a}(t,\xi). 
\]

Now let $(\pmb{v}_0(t,\xi),\ldots,\pmb{v}_{m-1}(t,\xi))$ be the fundamental 
matrix of \eqref{EQ}. This means, in particular, that 
$$
(\pmb{v}_0(0,\xi),\ldots,\pmb{v}_{m-1}(0,\xi))=I.
$$
Then each $\pmb{v}_j(t,\xi)$ can be represented by 
\[
\pmb{v}_j(t,\xi)=\mathscr{N}(t,\xi)^{-1}\Phi(t,\xi)\pmb{a}^j(t,\xi), 
\]
where $\pmb{a}^j(t,\xi)$ are the corresponding amplitude functions to 
$\pmb{v}_j(t,\xi)$. 
Since $\widehat{U}(t,\xi)=\sum_{j=0}^{m-1} 
\pmb{v}_j(t,\xi)\widehat{f}_j(\xi)$, 
we arrive at 
\[
\widehat{U}(t,\xi)=\sum_{j=0}^{m-1} \mathscr{N}(t,\xi)^{-1}\Phi(t,\xi)
\pmb{a}^j(t,\xi)\widehat{f}_j(\xi).
\]
Finally, let us find the estimates for the amplitude functions 
$\pmb{a}^j(t,\xi)$. Recalling that $\pmb{a}^j(t,\xi)$ satisfy the 
problem 
\[
D_t\pmb{a}^j=C(t,\xi)\pmb{a}^j \quad \text{with 
$(\pmb{a}^0(0,\xi),\cdots,\pmb{a}^{m-1}(0,\xi))=\mathscr{N}(0,\xi)$,}
\] 
we can write $\pmb{a}^j(t,\xi)$ by the Picard series: 
\begin{multline}\label{EQ:Picard}
\pmb{a}^j(t,\xi)= \\
\left(I+i\int^t_0 C(\tau_1,\xi) \, d \tau_1
+i^2\int^{t}_0 C(\tau_1,\xi) \, 
d \tau_1\int^{\tau_1}_0 
C(\tau_2,\xi)\, d \tau_2+\cdots \right) \pmb{a}^j(0,\xi).
\end{multline}
This implies that 
\begin{equation} \label{EQ:UNIF}
\left|\pmb{a}^j(t,\xi)\right| \le 
e^{c\int_\R 
\|\partial_\tau \mathscr{N}(\tau,\xi)\|_{L^\infty(\Rn)}\, d\tau} 
|\pmb{a}^j(0,\xi)|,
\end{equation}
where we have used the following fact: 

\medskip
\noindent
Let $f(t)$ be a $L^1_{\mathrm{loc}}$-function on $\mathbb{R}$. Then 
\[
e^{\int^{t}_{s} f(\tau) \, d \tau}
=1+\int^{t}_{s} f(\tau_1) \, d \tau_1
+\int^{t}_{s} f(\tau_1) \, 
d \tau_1 \int^{\tau_1}_{s} f(\tau_2) 
\, d \tau_2+\cdots.
\]
The proof of Proposition \ref{prop:Rep} is now finished. 
\end{proof}

 
\section{Proof of Theorem \ref{thm:Kirchhoff}}
\label{sec:3}
In this section we shall prove the global well-posedness 
for Kirchhoff system \eqref{EQ:Kirchhoff systems}. 
The strategy is to employ the Schauder-Tychonoff fixed point theorem. 
Let us consider the {\em linear} Cauchy problem \eqref{Eq} again: 
\[
\left\{
\begin{aligned}
& D_t U=A(t,D_x)U, \quad (t,x) \in \mathbb{R} \times \mathbb{R}^n,\\
& U(0,x)={}^T \left(f_0(x),f_1(x),\ldots,f_{m-1}(x)\right),
\end{aligned}\right.
\]
where $A(t,D_x)$ is the first order $m\times m$ pseudo-differential system,
with symbol $A(t,\xi)$ positively homogeneous of order one. We assume that 
$A(t,\xi)$ satisfies the regularity condition 
\eqref{hyp1} and the strictly 
hyperbolic condition \eqref{hyp2}--\eqref{hyp3}.
Notice that each characteristic root $\varphi_j(t,\xi)$ and 
its time derivative $\pa_t \varphi_j(t,\xi)$ are  positively
homogeneous of order one in $\xi$ on account of 
Proposition \ref{prop:root}. 
Furthermore, we observe from \eqref{EQ:Representation of U} and 
Plancherel's identity that if $U_0\in \mathbb{H}^\sigma(\Rn)$ for some 
$\sigma\ge0$, then the solution $U(t,x)$ to the linear equation \eqref{Eq} satisfies 
an energy estimate
\begin{equation} \label{EQ:Energy}
\|U(t,\cdot)\|_{\mathbb{H}^\sigma(\Rn)}\le C\|U_0\|_{\mathbb{H}^\sigma(\Rn)}, 
\quad \forall t\in \R.
\end{equation}

Let us introduce a class of symbols of differential operators, 
which is convenient for the fixed point argument. \\

\noindent 
{\bf Class $\mathscr{K}$.} {\em 
Given two constants $\Lambda>0$ and $K>0$, 
we say that a symbol $A(t,\xi)$    
of a pseudo-differential operator $A(t,D_x)$ 
belongs to 
$\mathscr{K}=\mathscr{K}(\Lambda,K)$ if $A(t,\xi/|\xi|)$ belongs to 
$\mathrm{Lip}_{\mathrm{loc}}(\mathbb{R};(L^\infty(\Rn\backslash0))^{m^2})$ and satisfies
\[
\|A(t,\xi/|\xi|)\|_{L^\infty(\R;(L^\infty(\Rn\backslash0))^{m^2})} \le \Lambda,
\]
\[
\int^\infty_{-\infty}
\left\| \partial_t A(t,\xi/|\xi|)\right\|_{(L^\infty(\Rn\backslash0))^{m^2}}\, dt
\le K. 
\] 
}

The next lemma is the heart of our argument. It will be applied with a 
sufficiently small constant $K_0>0$ which will be fixed later, and for which
all the constants in the estimates of the next lemma are positive.
\begin{lem}\label{lem:Kirchhoff}
Let $n\ge1$. Assume that the symbol $A(t,\xi)$ 
of a pseudo-differential operator $A(t,D_x)$ satisfies \eqref{hyp2}--\eqref{hyp3}
and 
belongs to $\mathscr{K}=\mathscr{K}(\Lambda,K)$ 
for some $\Lambda>0$ and $0<K\le K_0$ with a sufficiently small constant 
$K_0>0$.
Let $U\in C(\mathbb{R};\mathbb{L}^2(\Rn))$ 
be a solution to the Cauchy problem 
\[
D_tU=A(t,D_x)U, \quad U(0,x)=U_0(x)\in 
\mathbb{L}^2(\Rn) \cap \mathscr{Y}(\Rn),
\] 
and let $s(t)$ be the function
\[
s(t)=\langle SU(t,\cdot),U(t,\cdot)\rangle_{\mathbb{L}^2(\Rn)}. 
\]
Then there exist two constants $M>0$ and $c>0$
independent of $U$ and $K$  such that 
\begin{align}\label{core1}
& \|A(s(t),\omega)\|_{(L^\infty(\mathbb{S}^{n-1}))^{m^2}} \\
\le& \|A(s(0),\omega)\|_{(L^\infty(\mathbb{S}^{n-1}))^{m^2}}
+M \left(K\|U_0\|_{\mathbb{L}^2(\Rn)}^2
+\frac{1}{1-cK}\|U_0\|_{\mathscr{Y}(\Rn)}\right), 
\nonumber
\end{align}
\begin{multline}\label{core2}
\int^\infty_{-\infty}\left\|\partial_t \left[A(s(t),\omega)\right] 
\right\|_{(L^\infty(\mathbb{S}^{n-1}))^{m^2}} 
\, dt \\
\le M \left(K\|U_0\|_{\mathbb{L}^2(\Rn)}^2
+\frac{1}{1-cK}\|U_0\|_{\mathscr{Y}(\Rn)}\right).
\end{multline} 
\end{lem} 
We will be interested in sufficiently small $K_0>0$ so that we would have
$1-cK>0$ in the estimates above.
\begin{proof} The estimate \eqref{core1} easily follows from 
\eqref{core2} and the following identity:
\[
A(s(t),\omega)=A(s(0),\omega)+\int^t_0 \pa_\tau \left[A(s(\tau),\omega)\right]\, d\tau.
\]
Hence it is sufficient to 
concentrate on proving \eqref{core2}. However, since 
\[
\left\|\partial_t \left[ A(s(t),\omega) \right] 
\right\|_{(L^\infty(\mathbb{S}^{n-1}))^{m^2}} 
\le C|s^\prime(t)|,
\]
we only have  to show that
\begin{equation}\label{core3}
\int^\infty_{-\infty}|s^\prime(t)|\, dt
\le M\left(K\|U_0\|_{\mathbb{L}^2(\Rn)}^2
+\frac{1}{1-cK}\|U_0\|_{\mathscr{Y}(\Rn)}\right).
\end{equation} 
We recall from Proposition \ref{prop:Rep} that
\[
\widehat{U}(t,\xi)=\sum_{j=0}^{m-1} \mathscr{N}(t,\xi)^{-1}
\Phi(t,\xi)\pmb{a}^j(t,\xi)\widehat{f}_j(\xi), 
\] 
and its time-derivative version is given by 
\begin{multline*}
\widehat{U}^\prime(t,\xi)=\sum_{j=0}^{m-1} 
\{ \pa_t \mathscr{N}(t,\xi)^{-1}\Phi(t,\xi)\pmb{a}^j(t,\xi)
+\mathscr{N}(t,\xi)^{-1}\pa_t\Phi(t,\xi)
\pmb{a}^j(t,\xi)\\
+\mathscr{N}(t,\xi)^{-1}\Phi(t,\xi)\partial_t \pmb{a}^j(t,\xi)\}\widehat{f}_j(\xi).
\end{multline*}
Plugging these equations into $s^\prime(t)$, we can write 
\begin{equation}\label{EQ:s-prime}
s^\prime(t)=2 \mathrm{Re} \left\langle S\widehat{U}^\prime(t,\xi),
\widehat{U}(t,\xi) \right\rangle_{\mathbb{L}^2(\Rn)} 
=2\{I(t)+J(t)\}, 
\end{equation}
where
\begin{multline*}
I(t)=\\
\mathrm{Re} \sum_{j,k=0}^{m-1}
\left\langle
S \mathscr{N}(t,\xi)^{-1}\partial_t \Phi(t,\xi)\pmb{a}^j(t,\xi)\widehat{f}_j(\xi),
\mathscr{N}(t,\xi)^{-1}\Phi(t,\xi)\pmb{a}^k(t,\xi)\widehat{f}_k(\xi)
\right\rangle_{\mathbb{L}^2(\Rn)},
\end{multline*}
\begin{multline*}
J(t)=\\
\mathrm{Re} \sum_{j,k=0}^{m-1}
\left\langle
S \pa_t \mathscr{N}(t,\xi)^{-1} \Phi(t,\xi)\pmb{a}^j(t,\xi)\widehat{f}_j(\xi),
\mathscr{N}(t,\xi)^{-1}\Phi(t,\xi)\pmb{a}^k(t,\xi)\widehat{f}_k(\xi)
\right\rangle_{\mathbb{L}^2(\Rn)}\\
+\left\langle
S \mathscr{N}(t,\xi)^{-1} \Phi(t,\xi)\pa_t \pmb{a}^j(t,\xi)\widehat{f}_j(\xi),
\mathscr{N}(t,\xi)^{-1}\Phi(t,\xi)\pmb{a}^k(t,\xi)\widehat{f}_k(\xi)
\right\rangle_{\mathbb{L}^2(\Rn)}.
\end{multline*}
Since 
\[
\int^\infty_{-\infty} 
\left\{\left\|\partial_t \mathscr{N}(t,\xi)^{-1}\right\|_{(L^\infty(\Rn\backslash0))^{m^2}}
+\left\|\partial_t \pmb{a}^j(t,\xi)\right\|_{(L^\infty(\Rn\backslash0))^m}\right\}\, dt
\le CK
\]
on account of Proposition \ref{prop:Rep}, it follows that 
\begin{equation}\label{EQ:IND-TIME}
\int^\infty_{-\infty}|J(t)|\, dt \le CK\|U_0\|^2_{\mathbb{L}^2(\Rn)}
\end{equation}
with a certain constant $C>0$. 

It remains to estimate the oscillatory integral $I(t)$.
Writing 
\[
\text{$\pmb{a}^j(t,\xi)
={}^T(a^{1j}(t,\xi),\ldots,a^{mj}(t,\xi))$ and $\mathscr{N}(t,\xi)^{-1}=(n^{lp}(t,\xi))$,}
\]
we have 
\[
I(t)=\mathrm{Re} \sum_{j,k=0}^{m-1} \, \sum_{b,l,p,q=1}^m I_{j,k;b,l,p,q}(t),
\]
where 
\[
I_{j,k;b,l,p,q}(t)=
i\left\langle
s_{bl}n^{lp}(t,\xi)\varphi_p(t,\xi)e^{i\vartheta_p(t,\xi)}a^{pj}(t,\xi)\widehat{f}_j,
n^{bq}(t,\xi)e^{i\vartheta_q(t,\xi)}a^{qk}(t,\xi)\widehat{f}_k
\right\rangle_{L^2(\Rn)}
\]
with 
\[
\vartheta_p(t,\xi)=\int^t_0 \varphi_p(s,\xi)\, ds \quad (p=1,\ldots,m).
\]
Here the sum in $I_{j,k;b,l,p,q}(t)$ over $p=q$ does not contribute to $I(t)$. In fact, 
these integrals are pure imaginary. To see this fact, let us write
$$\varphi_p(t,\xi)=\varphi^+_p(t,\xi)-\varphi^-_p(t,\xi),$$ 
where $\varphi^+_p(t,\xi)$ and $\varphi^-_p(t,\xi)$ 
are positive and negative parts of $\varphi_p(t,\xi)$, respectively. 
Then we can write
\begin{align*}
& i^{-1}I_{j,k;b,l,p,p}(t)=\\
&\sum_\pm \pm \left\langle
s_{bl}n^{lp}(t,\xi)\sqrt{\varphi^\pm_p(t,\xi)}a^{pj}(t,\xi)\widehat{f}_j,n^{bp}(t,\xi)
\sqrt{\varphi^\pm_p(t,\xi)}a^{pk}(t,\xi)\widehat{f}_k
\right\rangle_{L^2(\Rn)}
\end{align*} 
and since $S$ is Hermitian, the sum
$
\sum_{j,k=0}^{m-1} \sum_{b,l,p=1}^m  i^{-1}I_{j,k;b,l,p,p}(t)
$
is real, and the real part 
$$
\mathrm{Re} \sum_{j,k=0}^{m-1} \sum_{b,l,p=1}^m I_{j,k;b,l,p,p}(t) =0
$$
vanishes. Therefore, by putting 
\[
\va_{pq}(t,\xi)=\sum_{b,l}s_{bl}n^{lp}(t,\xi)\varphi_p(t,\xi)
\overline{n^{bq}(t,\xi)},
\] 
which are positively homogeneous of order one in $\xi$, 
we can write 
\begin{align*}
I(t)=&\mathrm{Re} \sum_{p \ne q}\sum_{j,k} 
\left\langle
i\varphi_{pq}(t,\xi)e^{i\vartheta_p(t,\xi)}a^{pj}(t,\xi)\widehat{f}_j,
e^{i\vartheta_q(t,\xi)}a^{qk}(t,\xi)\widehat{f}_k
\right\rangle_{L^2(\Rn)}\\
=&-\mathrm{Im} \sum_{p \ne q} \sum_{j,k} \left\langle
\varphi_{pq}(t,\xi)e^{i\vartheta_p(t,\xi)}a^{pj}(t,\xi)\widehat{f}_j,
e^{i\vartheta_q(t,\xi)}a^{qk}(t,\xi)\widehat{f}_k
\right\rangle_{L^2(\Rn)}.
\end{align*}
Now let us consider the functional
\[
I_{p,q}(\eta(\cdot),t)=-\mathrm{Im} \sum_{j,k} 
\int_{\mathbb{S}^{n-1}}
I_{p,q,j,k}(\eta(\omega),t)\, d\sigma(\omega),
\]
where $\eta(\xi)$ is a function of homogeneous order zero, 
$d\sigma(\omega)$ is the $(n-1)$-dimensional Hausdorff measure, and we put 
\[
I_{p,q,j,k}(\eta(\omega),t)=
\int_0^\infty e^{i\eta(\omega) \rho}a^{pj}(t,\rho\omega)
\overline{a^{qk}(t,\rho\omega)}
\varphi_{pq}(t,\omega)
\widehat{f}_j(\rho\omega)
\overline{\widehat{f}_k(\rho\omega)}\rho^n\, d\rho.
\]
Furthermore, replacing $\eta(\cdot)$ in $I(\eta(\cdot),t)$ 
by a {\em real} parameter 
$\tau$, we define
\[
I^*_{p,q}(\tau)=\sup_{t\in \R}|I_{p,q}(\tau,t)|, \quad \tau\in \R.
\]
If we prove that 
\begin{equation}\label{EQ:Key est}
\sum_{p \ne q}\int^\infty_{-\infty}I^*_{p,q}(\tau)\, d\tau
\le \frac{C}{1-cK}\pmb{|}U_0\pmb{|}_{\mathscr{Y}(\Rn)},
\end{equation}
then we conclude that 
\begin{equation}\label{EQ:conclusion}
\int^\infty_{-\infty}|I(t)|\, dt
\le \frac{C}{1-cK}\pmb{|}U_0\pmb{|}_{\mathscr{Y}(\Rn)}.
\end{equation}
Indeed, since 
\begin{align*}
|I(t)|\le& C\sum_{p\ne q} \int_{\mathbb{S}^{n-1}}
|I_{p,q}(\vartheta_p(t,\omega)-\vartheta_q(t,\omega),t)|
\, d\sigma(\omega)\\
\le& C\sum_{p\ne q} \int_{\mathbb{S}^{n-1}}
I^*_{p,q}(\vartheta_p(t,\omega)-\vartheta_q(t,\omega))\, d\sigma(\omega),
\end{align*}
it follows from the Fubini-Tonnelli theorem that 
\[
\int^\infty_{-\infty}|I(t)|\, dt
\le C\sum_{p\ne q} \int_{\mathbb{S}^{n-1}}\left(
\int^\infty_{-\infty}I^*_{p,q}(\vartheta_p(t,\omega)-\vartheta_q(t,\omega))\, dt
\right)\, d\sigma(\omega).
\]
Here we note that 
\[
\inf_{t\in \R, \omega\in \mathbb{S}^{n-1}}|\va_p(t,\omega)
-\va_q(t,\omega)|\ge d(>0)
\qquad \text{for $p\ne q$,}
\]
for some $d>0$.
Then, by changing the variable $\tau_\omega
=\vartheta_{pq}(t,\omega)=\vartheta_p(t,\omega)-\vartheta_q(t,\omega)$, 
and by using \eqref{EQ:Key est},
we can estimate 
\begin{align*}
&\sum_{p\ne q}
\int_{\mathbb{S}^{n-1}}\left(
\int^\infty_{-\infty}I^*_{p,q}(\vartheta_p(t,\omega)-\vartheta_q(t,\omega))\, dt
\right)\, d\sigma(\omega)\\
= & \sum_{p\ne q} \int_{\mathbb{S}^{n-1}}\left(\int^\infty_{-\infty}
\frac{1}{\va_p(\vartheta^{-1}_{pq}(\tau_\omega,\omega))-
\va_q(\vartheta^{-1}_{pq}(\tau_\omega,\omega))}
I^*_{p,q}(\tau_\omega)\, d\tau_\omega
\right)\, d\sigma(\omega)\\
\le& d^{-1} \sum_{p\ne q} \int_{\mathbb{S}^{n-1}}\left(\int^\infty_{-\infty}
I^*_{p,q}(\tau_\omega)\, d\tau_\omega
\right)\, d\sigma(\omega)\\
\le& \frac{C}{1-cK}\pmb{|}U_0\pmb{|}_{\mathscr{Y}(\Rn)},
\end{align*}
which implies the estimate \eqref{EQ:conclusion}.

We now turn to prove the estimate \eqref{EQ:Key est}.
Recall the Picard series \eqref{EQ:Picard} for $a^{pj}(t,\xi)$; 
since $a^{pj}(0,\xi)=n_{pj}(0,\xi)$, it follows that 
\begin{multline*}
a^{pj}(t,\xi)=
n_{pj}(0,\xi)+i\int^t_0 c^p_{l_1}(\tau_1,\xi) n_{l_1 j}(0,\xi)\, d \tau_1\\
+i^2\int^{t}_0 c^p_{l_1}(\tau_1,\xi) \, 
d \tau_1\int^{\tau_1}_0 
c^{l_1}_{l_2}(\tau_2,\xi)n_{l_2 j}(0,\xi)\, d \tau_2+\cdots,
\end{multline*}
where each entry $c^l_k(t,\xi)$ of $C(t,\xi)$ is of the form 
$n^{lq}(t,\xi)\pa_t n_{pk}(t,\xi) e^{i\vartheta_{pq}(t,\xi)}$, and 
the behaviour of $n^{lq}(t,\xi)$ is similar to that of 
$n_{pk}(t,\xi)$. In the sequel 
we omit the indices of $\vartheta_{pq}(t,\xi)$, $n_{pk}(t,\xi)$ and 
$n^{lq}(t,\xi)$. Then 
$\vartheta(t,\xi)$ is positively homogeneous of order one and 
$n(t,\xi)$ is homogeneous of order zero in $\xi$ satisfying    
\begin{equation}\label{EQ:K-bound}
\int^\infty_{-\infty}
\|\pa_t n(t,\cdot)\|_{L^\infty(\Rn\backslash0)}
\, dt\le CK.
\end{equation}
Plugging these series into $I_{p,q,j,k}(\tau,t)$, we 
extract the integrals depending on $K^j$ for $j=0,1,2,\ldots$,
and estimate as follows: \\

\noindent 
(i) {\em Integrals independent of $K$} are easily handled by  the estimate
\[
\int^\infty_{-\infty}\left|\int_\Rn e^{i\tau|\xi|}
n(t,\xi)^2\va_{pq}(t,\xi)\widehat{f}_j(\xi)
\overline{\widehat{f}_k(\xi)}\, d\xi\right|
\, d\tau
\le C\pmb{|} U_0\pmb{|}_{\mathscr{Y}(\Rn)}.
\] 
Indeed, making the change of variable $\xi=\rho\omega$ 
$(\rho=|\xi|$, $\omega=\xi/|\xi|\in \mathbb{S}^{n-1}$),
the left-hand side becomes 
\begin{align*}
& \int^\infty_{-\infty}\left|\int_{\mathbb{S}^{n-1}}
n(t,\omega)^2\va_{pq}(t,\omega)
\left(\int^\infty_0 e^{i\tau \rho}
\widehat{f}_j(\rho\omega)
\overline{\widehat{f}_k(\rho\omega)}\rho^n \, d\rho \right)
d\sigma(\omega)\right|\, d\tau\\
\le& C\int^\infty_{-\infty}\left(\int_{\mathbb{S}^{n-1}}
\left|\int^\infty_0 e^{i\tau \rho}
\widehat{f}_j(\rho\omega)
\overline{\widehat{f}_k(\rho\omega)}\rho^n \, d\rho \right|
d\sigma(\omega)\right)\, d\tau
\le C\pmb{|} U_0\pmb{|}_{\mathscr{Y}(\Rn)},
\end{align*}
since $|n(t,\omega)^2\va_{pq}(t,\omega)|\le C$ for all $t\in \R$.

\bigskip 

\noindent 
(ii) {\em Integrals depending on $K$} are reduced to the 
following: 
\[
\int^\infty_{-\infty}\left|\int_\Rn e^{i\tau|\xi|}
i\left(\int^t_0 \pa_{\tau_1} 
n(\tau_1,\xi)e^{i\vartheta(\tau_1,\xi)} \, d \tau_1
\right)
n(0,\xi)\va_{pq}(t,\xi)\widehat{f}_j(\xi)
\overline{\widehat{f}_k(\xi)}\, d\xi\right|\, d\tau.
\]
Making the change of variable $\xi=\rho\omega$,
using the bounds $|n(0,\omega)\va_{pq}(t,\omega)|
\le C$ and Fubini's theorem, we estimate the above integrals as 
\[
C\int^\infty_{-\infty} 
\int_{\mathbb{S}^{n-1}} 
\left(\int^t_0 |\pa_{\tau_1} n(\tau_1,\omega)|
\left|
\int^\infty_0 e^{i(\tau+\vartheta(\tau_1,\omega))\rho}
\widehat{f}_j(\rho\omega)
\overline{\widehat{f}_k(\rho\omega)}\rho^n\, d\rho
\right|\, d \tau_1\right) \,
d\sigma(\omega) d\tau.
\]
Since $U_0\in \mathscr{Y}(\Rn)$, resorting to the invariance property of 
Lebesgue integrals with respect to the measure $d\tau$ 
and estimate \eqref{EQ:K-bound}, we conclude that 
the above integrals can be estimated as   
\begin{multline*}
C\int^t_0 
\|\pa_{\tau_1} n(\tau_1,\cdot)\|_{L^\infty(\mathbb{S}^{n-1})}\, d\tau_1 
\times \\
\int^\infty_{-\infty}\left(\int_{\mathbb{S}^{n-1}} 
\left|\int^\infty_0 e^{i\tau \rho}
\widehat{f}_j(\rho\omega)
\overline{\widehat{f}_k(\rho\omega)}\rho^n\, d\rho
\right| \, d\sigma(\omega)\right)\, d\tau\\
\le CK\pmb{|} U_0\pmb{|}_{\mathscr{Y}(\Rn)}.
\end{multline*}

\bigskip 
\noindent 
(iii) {\em In the integrals depending on $K^j$ for $j\ge2$}, 
the factors 
\[
e^{i\tau\rho}\underbrace{\int \cdots \int}_j \, \pa_{\tau_1} 
n_1(\tau_1,\omega)e^{i\vartheta_1(\tau_1,\rho\omega)}
\cdots
\pa_{\tau_j} n_j(\tau_j,\omega)e^{i\vartheta_j(\tau_j,\rho\omega)}\, d\tau_1 \cdots d\tau_j
\]
appear. Writing oscillatory factors as  
$e^{i(\tau+\vartheta(\tau_1,\omega)+\cdots+\vartheta(\tau_j,\omega))\rho}$, 
one can also handle such a factor by the invariance 
property of Lebesgue integrals. As a result, 
by using estimate \eqref{EQ:K-bound}, the present terms are bounded by 
$c^jK^j \pmb{|} U_0\pmb{|}_{\mathscr{Y}(\Rn)}$ for
some constant $c>0$.\\

Summing up these integrals and noting that $0<K<1$, we arrive at the estimate 
\eqref{EQ:Key est}, namely at
\[
\int^\infty_{-\infty}I^*(\tau)\, d\tau
\le C(1+cK+c^2K^2+\cdots)\pmb{|} U_0\pmb{|}_{\mathscr{Y}(\Rn)}
=\frac{C}{1-cK}\pmb{|} U_0\pmb{|}_{\mathscr{Y}(\Rn)},
\]
provided that $0< K\leq K_0$ for some sufficiently small constant
$K_0>0$.
In conclusion, by combining \eqref{EQ:IND-TIME} and \eqref{EQ:conclusion}, 
we get \eqref{core3}.
The proof of Lemma \ref{lem:Kirchhoff} is now finished. 
\end{proof}

\begin{proof}[Proof of Theorem \ref{thm:Kirchhoff}]
We employ the Schauder-Tychonoff fixed point theorem. Let 
$A(t,\xi) \in \mathscr{K}$, and we fix 
the data $U_0\in \mathbb{L}^2(\Rn)\cap \mathscr{Y}(\Rn)$. Then it follows 
from Lemma~\ref{lem:Kirchhoff} that the mapping 
\[ \Theta : A(t,\xi)\mapsto A(s(t),\xi) \] 
maps $\mathscr{K}=\mathscr{K}(\Lambda,K)$ into itself provided that 
$\|U_0\|^2_{\mathbb{L}^2(\Rn)}+\|U_0\|_{\mathscr{Y}(\Rn)}$ 
is sufficiently small, with
$\Lambda>2\| A(0,\xi/|\xi|)\|_{(L^\infty(\Rn\backslash 0))^{m^2}}$
and sufficiently small $0<K<K_0$.
Now $\mathscr{K}$ may be regarded as the convex subset 
of the Fr\'echet space $L^{\infty}_{\mathrm{loc}}(\R;(L^\infty(\Rn\backslash0))^{m^2})$, and we endow 
$\mathscr{K}$ with the induced topology. 
We shall show that $\mathscr{K}$ is 
compact in $L^{\infty}_{\mathrm{loc}}(\R;(L^\infty(\Rn\backslash0))^{m^2})$ and the mapping $\Theta$ 
is continuous on $\mathscr{K}$. \\

{\em Compactness of $\mathscr{K}$.} 
Since $\mathscr{K}$ is uniformly bounded and
equi-continuous on every compact $t$-interval, one can deduce from
the Ascoli-Arzel\`a theorem that $\mathscr{K}$
is relatively compact in
$L^{\infty}_{\mathrm{loc}}(\R;(L^\infty(\Rn\backslash0))^{m^2})$, and it is sequentially compact.
This means that every sequence $\{ A_j(t,\xi/|\xi|) \}_{j=1}^\infty$
in $\mathscr{K}$ has a subsequence, denoted by the same,
converging to some $A(\cdot,\xi/|\xi|)\in 
{\mathrm{Lip}}_{\mathrm{loc}}(\mathbb{R};(L^\infty(\Rn\backslash0))^{m^2})$:
\[
\left\{
\begin{aligned}
&A_j(t,\xi/|\xi|) \underset{(j\to\infty)}{\to}
A(t,\xi/|\xi|)
\quad
\text{in $L^{\infty}_{\mathrm{loc}}(\R;(L^\infty(\Rn\backslash0))^{m^2})$,}\\
& \|A(t,\xi/|\xi|)\|_{L^{\infty}_{\mathrm{loc}}(\R;(L^\infty(\Rn\backslash0))^{m^2})} \le \Lambda,
\end{aligned}\right.
\]
where we used the fact that the absolute continuity
of $\{A_j(t,\xi/|\xi|)\}$ is uniform in $j$
on account of the Vitali-Hahn-Saks theorem (see e.g.,
\S 2 in Chapter II from \cite{Yosida}),
since the finite limit
$\lim_{j\to \infty}\int^t_s \pa_\tau A_j(\tau,\xi/|\xi|)\, d\tau$ exists
for every interval $(s,t)$.
Moreover, the derivative $\pa_t A(t,\xi/|\xi|)$
exists almost everywhere on $\R$.
Now, for the derivative $\pa_t A(t,\xi/|\xi|)$,
if we prove that
\begin{equation}\label{EQ:limsup}
\int_{-\infty}^{+\infty}\|\pa_t A(t,\xi/|\xi|)\|_{(L^\infty(\Rn\backslash0))^{m^2}}
\, dt\le K,
\end{equation}
then $A(t,\xi/|\xi|) \in \mathscr{K}$, which proves the compactness of $\mathscr{K}$.

For the proof of the estimate \eqref{EQ:limsup}, we observe from Theorem 4 in \S1
of Chapter V of \cite{Yosida}
that the sequence 
$\{\pa_t A_j(\cdot,\xi/|\xi|)\}$ converges weakly to some matrix-valued function
$B(\cdot,\xi/|\xi|)\in L^1(\R;(L^\infty(\Rn\backslash0))^{m^2})$ as $j\to\infty$, since the finite limit
$\lim_{j\to \infty}\int^t_s \pa_\tau A_j(\tau,\xi/|\xi|)\, d\tau$ exists
for every interval $(s,t)$ and $\{\pa_t A_j(\cdot,\xi/|\xi|)\}$ is uniformly bounded
in $L^1(\R;(L^\infty(\Rn\backslash0))^{m^2})$:
\begin{equation}\label{EQ:Calculus}
\int_{-\infty}^\infty \|\pa_t A_j(t,\xi/|\xi|)\|_{(L^\infty(\Rn\backslash0))^{m^2}}\, dt\le K.
\end{equation}
By standard arguments we can
conclude that $\pa_t A(t,\xi/|\xi|)=B(t,\xi/|\xi|)$ for a.e. $t\in \R$.
Hence
\eqref{EQ:limsup} is true, since
\[
\int_{-\infty}^{+\infty}\|\pa_t A(t,\xi/|\xi|)\|_{(L^\infty(\Rn\backslash0))^{m^2}}\, dt
\le \liminf_{j\to\infty}
\int_{-\infty}^{+\infty}\|\pa_t A_j(t,\xi/|\xi|)\|_{(L^\infty(\Rn\backslash0))^{m^2}}\, dt\le K,
\]
where we used \eqref{EQ:Calculus}. \\

{\em Continuity of $\Theta$ on $\mathscr{K}$.} 
Let us take a sequence $\{ A_k(t,\xi/|\xi|) \}$ in 
$\mathscr{K}$ such that 
\begin{equation}\label{conv} 
\text{$A_k(t,\xi/|\xi|) \to A(t,\xi/|\xi|) \in \mathscr{K}$ 
\quad in 
$L^{\infty}_{\mathrm{loc}}(\R;(L^\infty(\Rn\backslash0))^{m^2})$ \quad $(k \to \infty)$,} 
\end{equation}
and let $U_k(t,x)$ and $U(t,x)$ be the corresponding solutions to $A_k(t,\xi)$ 
and $A(t,\xi)$, respectively, 
with fixed data $U_0$ satisfying the assumption 
of Theorem \ref{thm:Kirchhoff}. Put 
\begin{equation}\label{EQ:s-eq}
s_k(t)=\langle SU_k(t),U_k(t)\rangle_{\mathbb{L}^2(\Rn)}.
\end{equation}
Then we prove that the images 
\[
\text{$A_k(s_k(t),\xi)=\Theta (A_k(t,\xi))$ 
and $A(s(t),\xi)=\Theta(A(t,\xi))$}
\]
satisfy 
\begin{equation}\label{convergence} 
\text{$A_k(s_k(t),\xi/|\xi|) \to A(s(t),\xi/|\xi|)$ \quad in 
$L^{\infty}_{\mathrm{loc}}(\R;(L^\infty(\Rn\backslash0))^{m^2})$ \quad $(k \to \infty)$.} 
\end{equation} 
Using Proposition \ref{prop:Rep} again, 
we can write 
\[
U(t,x)=\sum_{j=0}^{m-1} 
\mathscr{F}^{-1} \left[\mathscr{N}(t,\xi)^{-1}\Phi(t,\xi)
\pmb{a}^j(t,\xi)\widehat{f}_j(\xi) \right](x), 
\]
\[
U_k(t,x)=\sum_{j=0}^{m-1} 
\mathscr{F}^{-1} \left[\mathscr{N}_k(t,\xi)^{-1}
\Phi_k(t,\xi)
\pmb{a}_k^j(t,\xi)\widehat{f}_j(\xi) \right](x).
\]
Notice that 
\begin{equation}\label{EQ:Phi-eq}
\Phi_k(t,\xi)\to \Phi(t,\xi) \quad 
\text{in $L^{\infty}_{\mathrm{loc}}(\R;(L^\infty(\Rn\backslash0))^{m^2})$} \quad (k\to\infty),
\end{equation}
\begin{equation}
\mathscr{N}_k(t,\xi)^{-1}\to 
\mathscr{N}(t,\xi)^{-1} \quad 
 \text{in $L^{\infty}_{\mathrm{loc}}(\R;(L^\infty(\Rn\backslash0))^{m^2})$} \quad 
(k\to\infty)
\end{equation}
on account of \eqref{conv}. Furthermore, we have  
\begin{equation}\label{EQ:Convergence}
\pmb{a}_k^j(t,\xi)\to 
\pmb{a}^j(t,\xi) \quad 
 \text{in $L^{\infty}_{\mathrm{loc}}(\R;(L^\infty(\Rn\backslash0))^m)$} \quad 
(k\to\infty).
\end{equation}
Indeed, we observe from previous argument that 
$\{\pa_t A_k(t,\xi/|\xi|)\}$ is weakly convergent to $\pa_t A(t,\xi/|\xi|)$
in $L^1(\R;(L^\infty(\Rn\backslash0))^{m^2})$, and hence, 
$\{\pa_t \mathscr{N}_k(t,\xi)\}$ is also weakly convergent to 
$\pa_t \mathscr{N}(t,\xi/|\xi|)$
in $L^1(\R;(L^\infty(\Rn\backslash0))^{m^2})$. 
Thus we find from this observation and the Picard series \eqref{EQ:Picard} for 
$\pmb{a}_k^j(t,\xi)$ that the convergence \eqref{EQ:Convergence} is proved. 
Then, by using the Lebesgue dominated convergence theorem, 
we conclude from \eqref{EQ:s-eq}--\eqref{EQ:Convergence} 
that $s_k(t)\to s(t)$ $(k\to\infty)$, which implies 
\eqref{convergence}. \\

{\em Completion of the proof of Theorem \ref{thm:Kirchhoff}.} 
By using the Schauder--Tychonoff fixed point theorem, we can show 
that $\Theta$ has a fixed point in $\mathscr{K}$, 
with $K_0>0$ in Lemma \ref{lem:Kirchhoff} sufficiently
small, so that constants in \eqref{core1}-\eqref{core2} are positive.
Hence, we conclude that if $U_0\in \mathbb{L}^2(\Rn)\cap \mathscr{Y}(\Rn)$,
then the solutions $U(t,x)$ of 
$$D_tU=A(t,D_x)U$$ with the Cauchy data $U(0,x)=U_0(x)$ are solutions to the nonlinear system  
\eqref{EQ:Kirchhoff systems} and belong to $C(\R;\mathbb{L}^2(\Rn))$. 
Furthermore, these solutions $U$ satisfy the energy estimates \eqref{EQ:Energy}.

Finally, we prove the uniqueness.
Let $U,V$ be two solutions 
to the nonlinear system \eqref{EQ:Kirchhoff systems} with 
$U(0,x)=V(0,x)=U_0(x)$, 
and let  
$$
s_U(t)=\langle SU(t),U(t)\rangle_{\mathbb{L}^2(\Rn)}\quad\textrm{ and }\quad
s_V(t)=\langle SV(t),V(t)\rangle_{\mathbb{L}^2(\Rn)}$$
be the corresponding nonlocal terms, respectively. In this time, we need to assume that 
$U_0\in\mathbb{H}^1(\Rn)$.
We observe from Proposition \ref{prop:Rep} and the property of the mapping 
$\Theta$ in the fixed point argument that, when we consider the integral representations of 
$U$ and $V$,  the functions 
$\Phi(t,\xi)$ and $\mathscr{N}(t,\xi)$ in the representation \eqref{EQ:Representation of U}
may be replaced by 
\[
\text{$\Phi(s(t),\xi)=\mathrm{diag}\left(e^{i\int^t_0 \varphi_1(s(\tau),\xi)\, d\tau},\cdots,
e^{i\int^t_0 \varphi_m(s(\tau),\xi)\, d\tau} \right)$
and $\mathscr{N}(s(t),\xi)$,}
\]
respectively, where $s(t)=s_U(t)$ or $s_V(t)$. Since $\pmb{a}^j(t,\xi)$ 
are solutions to the linear system 
$D_t\pmb{a}^j(t,\xi)=C(t,\xi)\pmb{a}^j(t,\xi)$, 
where 
$$C(t,\xi)=\Phi(t,\xi)^{-1} (D_t \mathscr{N}(t,\xi))
\mathscr{N}(t,\xi)^{-1}\Phi(t,\xi),
$$
the amplitude functions $\pmb{a}^j(s,\xi)$ 
for nonlinear system  
satisfy ordinary differential systems
\begin{equation}\label{EQ:Amp-nonlinear}
D_t\pmb{a}^j(s(t),\xi)=C(s(t),\xi)\pmb{a}^j(s(t),\xi),
\end{equation}
where 
$$C(s(t),\xi)=\Phi(s(t),\xi)^{-1} (D_t \mathscr{N}(s(t),\xi))
\mathscr{N}(s(t),\xi)^{-1}\Phi(s(t),\xi)
$$
are in $L^1(\R;(L^\infty(\Rn\backslash0))^{m^2})$.
Thus the solutions $U,V$ of  \eqref{EQ:Kirchhoff systems} have 
the following forms:
\begin{equation*} 
U(t,x)=\sum_{j=0}^{m-1} 
\mathscr{F}^{-1} \left[\mathscr{N}(s_U(t),\xi)^{-1}\Phi(s_U(t),\xi)
\pmb{a}^j(s_U(t),\xi)\widehat{f}_j(\xi) \right](x), 
\end{equation*}
\begin{equation*} 
V(t,x)=\sum_{j=0}^{m-1} 
\mathscr{F}^{-1} \left[\mathscr{N}(s_V(t),\xi)^{-1}\Phi(s_V(t),\xi)
\pmb{a}^j(s_V(t),\xi)\widehat{f}_j(\xi) \right](x).
\end{equation*}
Then we can write 
\begin{align} \label{EQ:difference}
&\|U(t)-V(t)\|^2_{\mathbb{L}^2(\Rn)}\\
=&\sum_{j,k=0}^{m-1}\int_\Rn \left(\pmb{b}^j(s_U(t),\xi)-\pmb{b}^j(s_V(t),\xi)\right)
\cdot \overline{\left(\pmb{b}^k(s_U(t),\xi)-\pmb{b}^k(s_V(t),\xi)\right)}\widehat{f}_j(\xi)
\overline{\widehat{f}_k(\xi)}\, d\xi,
\nonumber
\end{align}
where we put
\[
\pmb{b}^j(s_U(t),\xi)=\mathscr{N}(s_U(t),\xi)^{-1}\Phi(s_U(t),\xi)
\pmb{a}^j(s_U(t),\xi),
\]
\[
\pmb{b}^j(s_V(t),\xi))
=\mathscr{N}(s_V(t),\xi)^{-1}\Phi(s_V(t),\xi)
\pmb{a}^j(s_V(t),\xi).
\]
The functional $s_U(t)$ is Lipschitz with respect to $U$ 
since 
\begin{align}\nonumber
|s_U(t)-s_V(t)|\le& 
|\langle S(U(t)-V(t)),U(t)\rangle_{\mathbb{L}^2(\Rn)}|
+|\langle SV(t),U(t)-V(t)\rangle_{\mathbb{L}^2(\Rn)}|\\
\le& C\|U_0\|_{\mathbb{L}^2(\Rn)}\|U(t)-V(t)\|_{\mathbb{L}^2(\Rn)}.
\label{EQ:Lip1}
\end{align}
Since $A(s,\xi/|\xi|)$ is Lipschitz with respect to $s$, 
$\mathscr{N}(s,\xi)^{-1}$ and $\va_k(s,\xi)$ 
also depend on $s$ Lipschitz continuously; thus we find 
from \eqref{EQ:Lip1} that 
\begin{align} \label{EQ:N-Lip}
& \left\|\mathscr{N}(s_U(t),\xi)^{-1}-\mathscr{N}(s_V(t),\xi)^{-1}
\right\|_{(L^\infty(\Rn\backslash0))^{m^2}}\\
\le& C|s_U(t)-s_V(t)|
\le C\|U_0\|_{\mathbb{L}^2(\Rn)}\|U(t)-V(t)\|_{\mathbb{L}^2(\Rn)},
\nonumber
\end{align}
\begin{align} \label{EQ:Ph-Lip}
& \left\|\Phi(s_U(t),\xi)-\Phi(s_V(t),\xi)\right\| 
\le \sum_{k=1}^m \left| e^{i\int_0^t\varphi_k(s_U(\tau),\xi)\, d\tau}-
e^{i\int_0^t\varphi_k(s_V(\tau),\xi)\, d\tau}\right|
\\
\le& \sum_{k=1}^m\int_0^t |\varphi_k(s_U(\tau),\xi)-\varphi_k(s_V(\tau),\xi)|
\, d\tau
\nonumber\\
\le& C|\xi|\int^t_0 |s_U(\tau)-s_V(\tau)|\, d\tau
\nonumber \\
\le& C|\xi|\|U_0\|_{\mathbb{L}^2(\Rn)}
\int^t_0 \|U(\tau)-V(\tau)\|_{\mathbb{L}^2(\Rn)}\, d\tau,
\nonumber
\end{align}
where $\|\cdot \|$ denotes a matrix norm.
Furthermore, the amplitude functions $\pmb{a}^j(s,\xi)$ satisfy the 
following estimates:
\begin{equation}\label{EQ:a-Lip}
\|\pmb{a}^j(s_U(t),\xi)-\pmb{a}^j(s_V(t),\xi)\|_{(L^\infty(\Rn\backslash0))^m}
\le C\|U_0\|_{\mathbb{L}^2(\Rn)}\|U(t)-V(t)\|_{\mathbb{L}^2(\Rn)}.
\end{equation}
In fact, since $\pmb{a}^j(s,\xi)$ satisfy the ordinary differential system 
\eqref{EQ:Amp-nonlinear} with 
$C(s,\xi)\in L^1((0,\delta);(L^\infty(\Rn\backslash0))^{m^2})$,
it follows that 
$$D_s\pmb{a}^j(s,\xi)=C(s,\xi)\pmb{a}^j(s,\xi),$$
and hence, 
$\pmb{a}^j(s,\xi)$ are Lipschitz in $s$. Therefore, there exists a constant $L>0$ 
such that 
\begin{align*}
& \|\pmb{a}^j(s_U(t),\xi)-\pmb{a}^j(s_V(t),\xi)\|_{(L^\infty(\Rn\backslash0))^m}\\
\le& L|s_U(t)-s_V(t)|\\
\le& C\|U_0\|_{\mathbb{L}^2(\Rn)}\|U(t)-V(t)\|_{\mathbb{L}^2(\Rn)},
\end{align*}
where we used \eqref{EQ:Lip1} in the last step. This proves \eqref{EQ:a-Lip}.
Summarising \eqref{EQ:N-Lip}--\eqref{EQ:a-Lip}, we conclude that
\begin{align*}
& |\pmb{b}^j(s_U(t),\xi)-\pmb{b}^j(s_V(t),\xi)|\\
\le& C\|U_0\|_{\mathbb{L}^2(\Rn)}\left(\|U(t)-V(t)\|_{\mathbb{L}^2(\Rn)}
+|\xi|
\int^t_0 \|U(\tau)-V(\tau)\|_{\mathbb{L}^2(\Rn)}\, d\tau \right).
\end{align*}
Thus, \eqref{EQ:difference} together with these estimates imply that
\begin{multline*}
\|U(t)-V(t)\|^2_{\mathbb{L}^2(\Rn)}
\le C\Bigg\{ \|U_0\|^4_{\mathbb{L}^2(\Rn)}\|U(t)-V(t)\|^2_{\mathbb{L}^2(\Rn)}\\
\left. +\|U_0\|^2_{\mathbb{L}^2(\Rn)}\|D_x U_0\|^2_{\mathbb{L}^2(\Rn)}\left(\int^t_0 \|U(\tau)-V(\tau)\|_{\mathbb{L}^2(\Rn)}
\, d\tau\right)^2\right\}.
\end{multline*}
Since $\|U_0\|_{\mathbb{L}^2(\Rn)}$ is sufficiently small, we obtain 
\[
\|U(t)-V(t)\|_{\mathbb{L}^2(\Rn)}
\le C(\|U_0\|_{\mathbb{L}^2(\Rn)})
\|D_x U_0\|_{\mathbb{L}^2(\Rn)}\int^t_0 \|U(\tau)-V(\tau)\|_{\mathbb{L}^2(\Rn)}\, d\tau
\]
for some function $C(\|U_0\|_{\mathbb{L}^2(\Rn)}).$
Thus, applying Gronwall's lemma to the above inequality, we conclude that
$U(t)=V(t)$ for all $t\in\R$. This proves the uniqueness of solutions. 
The proof of Theorem \ref{thm:Kirchhoff} is now finished. 
\end{proof}


\section{A final remark}
Observing the inclusion \eqref{EQ:weight}, we can also prove:
\begin{thm}\label{thm:weight-Kirchhoff}
Let $n\ge1$ and $\varkappa\in (1,n+1]$. 
Assume that system \eqref{EQ:Kirchhoff systems} is strictly hyperbolic, and that 
$A(s,\xi)=(a_{jk}(s,\xi))_{j,k=1}^m$ is an $m\times m$ matrix, positively 
homogeneous of order one in $\xi$, whose entries $a_{jk}(s,\xi)$ satisfy 
$|\xi|^{-1+|\alpha|} \pa^\alpha_\xi a_{jk}(s,\xi)\in 
\mathrm{Lip}([0,\delta];L^\infty(\Rn\backslash0))$ for any $0\le |\alpha|\le [\varkappa]+1$ and 
for some $\delta>0$. 
If $U_0$ are small in the space $\mathbb{H}^1_\varkappa(\mathbb{R}^n)$,
then system \eqref{EQ:Kirchhoff systems} has a unique solution 
$U\in C(\mathbb{R};\mathbb{H}^1(\mathbb{R}^n))\cap 
C^1(\mathbb{R};\mathbb{L}^2(\mathbb{R}^n))$.
\end{thm}

\noindent 
{\em Outline of the proof.} In order to prove the theorem, let us introduce a 
subclass of $\mathscr{K}$ as follows: \\

\noindent 
{\bf Class $\mathscr{K}^\prime$.} {\em 
Given three constants $\Lambda>0$, $K>0$ and $\varkappa>1$, 
we say that the symbol $A(t,\xi)$    
of a pseudo-differential operator $A(t,D_x)$ 
belongs to 
$\mathscr{K}^\prime=\mathscr{K}^\prime(\Lambda,K,\varkappa)$ if $A(t,\xi)$ belongs to 
$\mathrm{Lip}_{\mathrm{loc}}(\mathbb{R};(C^{[\varkappa]+1}(\Rn\backslash0))^{m^2})$ and satisfies
\[
\|A(t,\xi/|\xi|)\|_{L^\infty(\R;(L^\infty(\Rn\backslash0))^{m^2})} \le \Lambda,
\]
\[
\left\| |\xi|^{-1+|\alpha|}\pa^\alpha_\xi \partial_t A(t,\xi)
\right\|_{(L^\infty(\Rn\backslash0))^{m^2}}
\le C_\alpha K\langle t\rangle^{-\varkappa}, 
\quad 0\le \forall |\alpha|\le [\varkappa]+1. 
\] 
}

We have the following lemma:
\begin{lem}\label{lem:weight-Kirchhoff}
Let $n\ge1$ and $1<\varkappa\le n+1$. Assume that the symbol $A(t,\xi)$ 
of a differential operator $A(t,D_x)$ satisfies \eqref{hyp2}--\eqref{hyp3}
and 
belongs to $\mathscr{K}^\prime$ 
for some $\Lambda>0$ and $0<K<K_0$
with sufficiently small $K_0$.
Let $U\in C(\mathbb{R};\mathbb{H}^1(\Rn))\cap C^1(\mathbb{R};\mathbb{L}^2(\Rn))$ 
be a solution to the Cauchy problem 
\[
D_tU=A(t,D_x)U, \quad U(0,x)=U_0(x)\in 
\mathbb{H}^1_\varkappa(\Rn),
\] 
and let $s(t)$ be the functional 
\[
s(t)=\langle SU(t,\cdot),U(t,\cdot)\rangle_{\mathbb{L}^2(\Rn)}. 
\]
Then, for any $0\le |\alpha|\le [\varkappa]+1$, 
there exist  constants $M_\alpha>0$ and
$c_{\alpha,\varkappa}>0$ independent of $U$ such that 
\begin{align*}
& \left\|A(s(t),\xi/|\xi|)\right\|_{(L^\infty(\Rn\backslash0))^{m^2}} \\
\le& \left\|A(s(0),\xi/|\xi|)\right\|_{(L^\infty(\Rn\backslash0))^{m^2}}
+M_0 \left(K\|U_0\|_{\mathbb{L}^2(\Rn)}^2
+\frac{1}{1-c_{\alpha,0} K}\|U_0\|^2_{\mathbb{H}^1_\varkappa(\Rn)}\right), 
\end{align*}
\begin{multline*}
\left\| |\xi|^{-1+|\alpha|}\pa^\alpha_\xi \partial_t A(s(t),\xi)
\right\|_{(L^\infty(\Rn\backslash0))^{m^2}}\\
\le M_\alpha \left(K\|U_0\|_{\mathbb{L}^2(\Rn)}^2
+\frac{1}{1- c_{\alpha,\varkappa} K}\|U_0\|^2_{\mathbb{H}^1_\varkappa(\Rn)}\right)
\langle t\rangle^{-\varkappa}.
\end{multline*} 
\end{lem} 
 
The proof of Lemma \ref{lem:weight-Kirchhoff} can be done by 
some modifications of the argument of Lemma \ref{lem:Kirchhoff},
and we take $K_0>0$ small enough so that $1-c_{\alpha,\varkappa} K_0>0$
for all $\alpha$ and $\varkappa$.
The following lemma can be obtained by a similar proof as 
Lemma A.1 of D'Ancona \& Spagnolo \cite{Dancona2} 
(see also Lemma 3.2 of \cite{Ext-Matsuyama}, and \cite{Yamazaki}).
\begin{lem}\label{lem:Spagnolo}
Let $n\ge1$ and $\varkappa\in (1,n+1]$. Assume that 
$\va(\xi)\in C(\Rn\backslash0)$ is a 
positively homogeneous function of order one.
Then 
\[
\int_{\mathbb{S}^{n-1}}\left|
\int^\infty_0 e^{i\tau\rho}\hat{f}_1(\rho\omega)
\hat{f}_2(\rho\omega)\va(\omega)\rho^n\, d\rho\right|\, d\sigma(\omega)
\le C_\varkappa\langle \tau \rangle^{-\varkappa}
\|f_1\|_{H^1_\varkappa(\Rn)}\|f_2\|_{H^1_\varkappa(\Rn)}
\]
for any $f_1,f_2\in \mathscr{S}(\Rn)$, where $\tau$ is a real parameter.
\end{lem}

In particular, observing the proof of Lemma \ref{Diagonalisation} 
(see Proposition 6.4 in \cite{Mizohata}), one can check that 
if $A(t,\xi)\in \mathscr{K}^\prime$, then 
derivatives of $\mathscr{N}(t,\xi)=(n_{jk}(t,\xi))_{j,k=1}^m$ and 
$\mathscr{N}(t,\xi)^{-1}=(n^{pq}(t,\xi))_{p,q=1}^m$ satisfy  
\[
\left\||\xi|^{-|\alpha|}\pa^\alpha_\xi n_{jk}(t,\xi)\right\|_{L^\infty(\Rn\backslash0)}, 
\, \left\||\xi|^{-|\alpha|}\pa^\alpha_\xi n^{pq}(t,\xi)\right\|_{L^\infty(\Rn\backslash0)}
\le C_\alpha \Lambda,
\]
\[
\left\||\xi|^{-|\alpha|}\pa^\alpha_\xi \pa_t n_{jk}(t,\xi)\right\|_{L^\infty(\Rn\backslash0)}, 
\, \left\||\xi|^{-|\alpha|}\pa^\alpha_\xi \pa_t n^{pq}(t,\xi)\right\|_{L^\infty(\Rn\backslash0)}
\le C_\alpha K\langle t \rangle^{-\varkappa},
\]
for any $0\le |\alpha|\le [\varkappa]+1$. Furthermore, $t$-derivatives of 
amplitudes $\pmb{a}^j(t,\xi)$ 
are estimated by 
\[
\|\pa_t\pmb{a}^j(t,\xi)\|_{(L^\infty(\Rn\backslash0))^m}\le CK\langle t \rangle^{-\varkappa}.
\]
Combining these estimates and the decay estimates for oscillatory integrals 
given in Lemma \ref{lem:Spagnolo}, we can perform the integration by parts with 
respect to $\rho=|\xi|$ in oscillatory integrals $I_{p,q}(\tau,t)$. Thus we find that 
\[
\sum_{p\ne q}I^*_{p,q}(\tau)\le \frac{C_\varkappa}{1-
c_{\alpha,\varkappa} K}\|U_0\|^2_{\mathbb{H}^1_\varkappa(\Rn)}
\langle \tau\rangle^{-\varkappa}
\] 
for any $\varkappa\in (1,n+1]$, which implies that 
\[
|I(t)|\le \frac{C_\varkappa}{1-
c_{\alpha,\varkappa} K}\|U_0\|^2_{\mathbb{H}^1_\varkappa(\Rn)}
\langle t\rangle^{-\varkappa}.
\] 
As to $J(t)$, we easily obtain  
\[
|J(t)|\le CK\|U_0\|^2_{\mathbb{L}^2(\Rn)}\langle t \rangle^{-\varkappa}.
\]
Hence Lemma \ref{lem:weight-Kirchhoff} is proved by combining the decay estimates 
for $I(t)$ and $J(t)$. \\

Resorting to Lemma \ref{lem:weight-Kirchhoff}, we can perform the fixed point 
argument as in the previous section (see also \cite{Ext-Matsuyama}), 
and as a result, 
the solution $U(t,x)$ to the linear system will be, of course, a solution to the original 
nonlinear system, which allows us to conclude the proof of Theorem 
\ref{thm:weight-Kirchhoff}.  \hfill $\Box$

\end{document}